\documentclass[10pt]{amsart}

\usepackage[dvips]{graphicx}
\usepackage{amsthm, amssymb, wrapfig}
\usepackage[all]{xy}
\usepackage{longtable, booktabs, ascmac}
\DeclareMathOperator{\rank}{rank}

\DeclareMathOperator{\Isom}{Isom}
\DeclareMathOperator{\Hom}{Hom}
\DeclareMathOperator{\Sign}{Sign}

\newcommand{\id}{\mathrm{id}}
\newcommand{\rmO}{\mathrm{O}}

\theoremstyle{plain}
\newtheorem{theorem}{Theorem}[section]
\newtheorem{proposition}[theorem]{Proposition}
\newtheorem{lemma}[theorem]{Lemma}
\newtheorem{corollary}[theorem]{Corollary}

\theoremstyle{definition}
\newtheorem{definition}[theorem]{Definition}

\theoremstyle{remark}
\newtheorem{remark}[theorem]{Remark}

\renewcommand{\epsilon}{\varepsilon}
\renewcommand{\phi}{\varphi}

\numberwithin{equation}{section}

\begin{document}

%\doublespacing

\title[Classification of involutions on Enriques surfaces]{Classification of involutions on Enriques surfaces}
\author[H. Ito]{Hiroki Ito}
\author[H. Ohashi]{Hisanori Ohashi}
\address{
Graduate School of Mathematics, Nagoya University, 
Furo-cho, Chikusa-ku, Nagoya 464-8602 JAPAN}
\email{m04003w@math.nagoya-u.ac.jp}
\address{Department of Mathematics, 
Faculty of Science and Technology, 
Tokyo University of Science, 
2641 Yamazaki, Noda, 
Chiba 278-8510, JAPAN}
\email{ohashi@ma.noda.tus.ac.jp}
\thanks{Supported in part by 
the JSPS Grant-in-Aid for Scientific Research (S) 22224001 and for Young Scientists (B) 23740010.}
\subjclass[2000]{14J28}
\date{\today}

\begin{abstract}
We present the classification of involutions on Enriques surfaces. 
We classify those into $18$ types with the help of the lattice theory due to Nikulin. 
We also give all examples of the classification. 
\end{abstract}

\maketitle

\section{Introduction}

An Enriques surface $Y$ is a compact complex surface satisfying the following conditions: 
\begin{enumerate}
\item the geometric genus and the irregularity vanish, 
\item the bi-canonical divisor on $Y$ is linearly equivalent to $0$. 
\end{enumerate}
Every Enriques surface $Y$ is a quotient of a $K3$ surface $X$ by a fixed point free involution $\epsilon$. 
In this work, we give the classification of involutions on Enriques surfaces. 

An involution $\iota$ on $Y$ lifts to two involutions of $X$. 
One of them, which we denote by $g$, acts on $H^0(X, \Omega^2)$ trivially. 
An involution with this property is called symplectic or Nikulin involution. 
To classify $\iota$, we study the pair of involutions $(g, \epsilon)$. 
For our purpose, we use the theory of the classification of involutions of a lattice with condition on a sublattice, due to V.~V.~Nikulin~\cite{Nik4}. 

Let $S$ be a lattice and $\theta$ an involution of $S$. 
In \cite{Nik4}, the determining condition of a triple $(L, \phi, i)$ with the condition $(S, \theta)$ satisfying the following commutative diagram is given: 
\[\xymatrix{
L \ar[r]^{\phi} & L \\
S \ar[r]_{\theta} \ar@{^{(}->}[u]^{i} & S \ar@{_{(}->}[u]_{i} 
}\]
Here $L$ is a unimodular lattice, $\phi$ is an involution of $L$, and $i\colon S\to L$ is a primitive embedding. 
To investigate $(L, \phi, i)$, we use the following invariants: 
Let $L_{\pm}=\{x\in L \mid \phi(x) = \pm x\}$ and $S_{\pm}=\{x\in S \mid \theta(x)=\pm x\}$. 
From the primitive embedding $i\colon S\to L$, we get primitive embeddings $i_{\pm}\colon S_{\pm}\to L_{\pm}$. 
Hence we have the orthogonal complements $K_{\pm}=(S_{\pm})^{\perp}_{L_{\pm}}$ and images of projection 
\begin{align*}
H_-&=p_{S_-}((L\cap (L_+\oplus S_-)\otimes \mathbb{Q})/L_+\oplus S_-) \subset A_{S_-}, \\
\widetilde{H_-}&=p_{S_-}((L\cap (K_+\oplus S_-)\otimes \mathbb{Q})/K_+\oplus S_-) \subset H_-, 
\end{align*}
where $A_{S_-}$ is the discriminant group of $S_-$. 

We apply this theory as $L=H^2(X, \mathbb{Z})$, $S=\{x\in H^2(X, \mathbb{Z}) \mid g^*(x)=-x\}$ and $\phi=\epsilon^*$. 
Next theorem is our main result. 
\begin{theorem}\label{MainThm}
Involutions of Enriques surfaces are classified as follows: 

\setlongtables
\begin{longtable}{cllllc}
\caption{Invariants and the model}\label{table1} \\ \toprule
\text{No.} & $S_+(\frac{1}{2})$ & $S_-(\frac{1}{2})$ & $q_{S_-}|_{H_-}$ & $q_{S_-}|_{\widetilde{H_-}}$ & Horikawa model\\ \toprule 
\endfirsthead
\toprule
\text{No.} & $S_+(\frac{1}{2})$ & $S_-(\frac{1}{2})$ & $q_{S_-}|_{H_-}$ & $q_{S_-}|_{\widetilde{H_-}}$ & Horikawa model\\ \toprule \endhead \endfoot
\textup{[1]} & $\{0\}$ & $E_8$ & $u^4$ & & \parbox{35pt}{\includegraphics[scale=0.2]{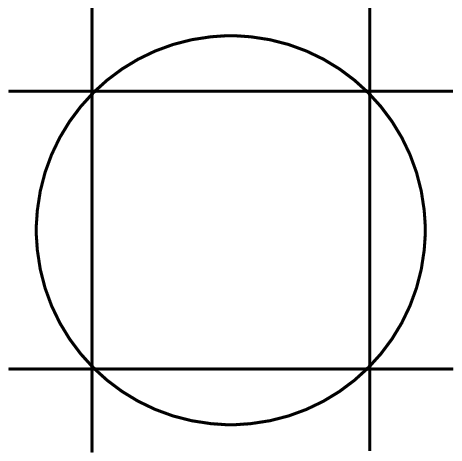}}\\ 
\textup{[2]} & $\{0\}$ & $E_8$ & $u^3\oplus w$ & & \parbox{35pt}{\includegraphics[scale=0.2]{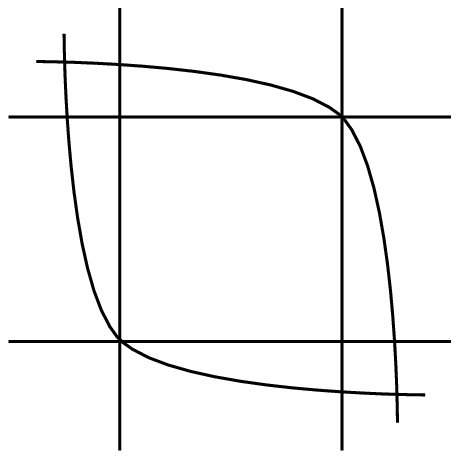}}\\ 
\textup{[3]} & $\{0\}$ & $E_8$ & $u^3\oplus z$ & & \parbox{35pt}{\includegraphics[scale=0.2]{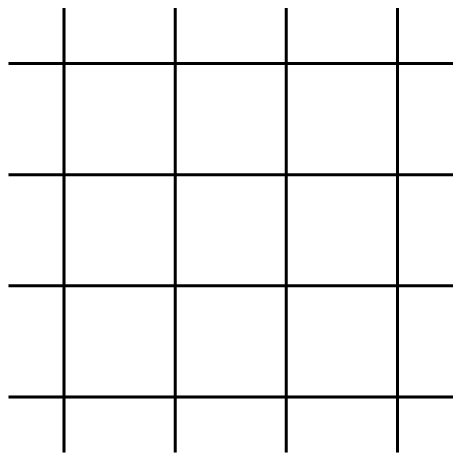}}\\ \midrule
\textup{[4]} & $A_1$ & $E_7$ & $u^3\oplus w$ & & \parbox{35pt}{\includegraphics[scale=0.2]{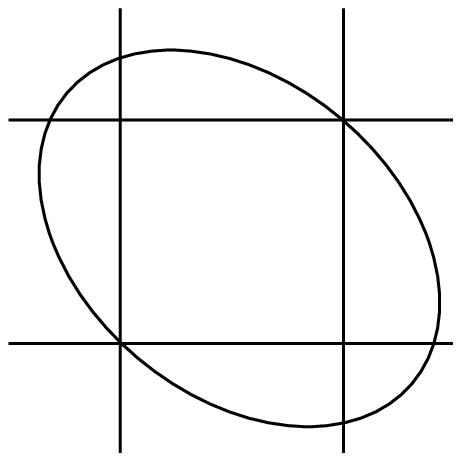}}\\ 
\textup{[5]} & $A_1$ & $E_7$ & $u^2\oplus w^2$ & & \parbox{35pt}{\includegraphics[scale=0.2]{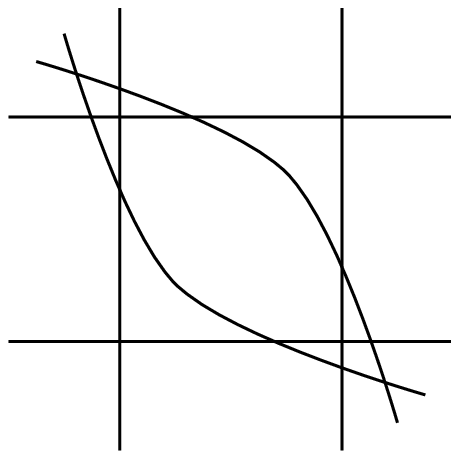}}\\ \midrule
\textup{[6]} & $A_1^2$ & $D_6$ & $u^2\oplus w^2$ & & \parbox{35pt}{\includegraphics[scale=0.2]{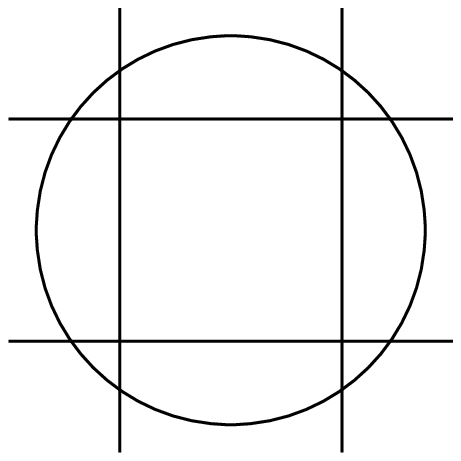}}\\ 
\textup{[7]} & $A_1^2$ & $D_6$ & $u\oplus w^3$ & & \parbox{35pt}{\includegraphics[scale=0.2]{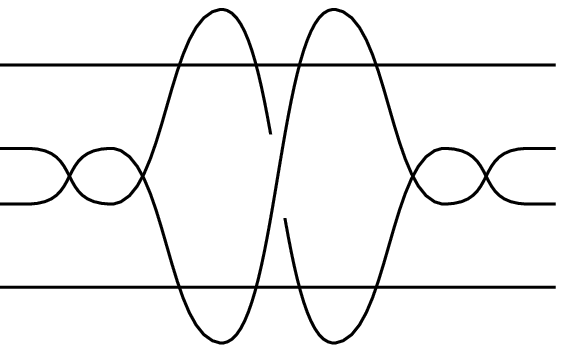}}\\ \midrule
\textup{[8]} & $A_1^3$ & $D_4\oplus A_1$ & $u\oplus w^3$ & & \parbox{35pt}{\includegraphics[scale=0.2]{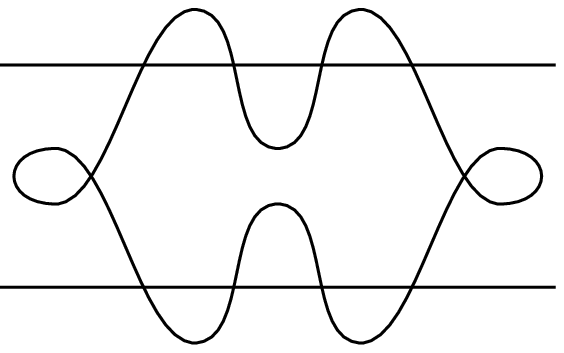}}\\ 
\textup{[9]} & $A_1^3$ & $D_4\oplus A_1$ & $w^4$ & & \parbox{35pt}{\includegraphics[scale=0.2]{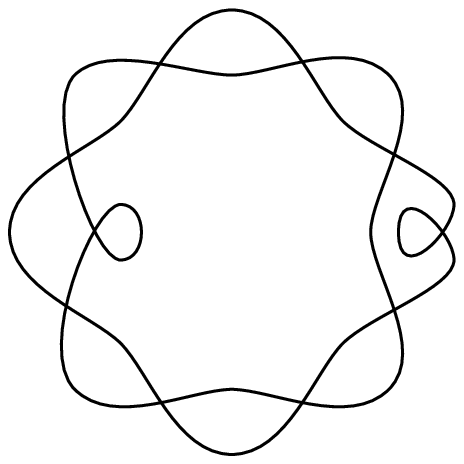}}\\ \midrule
\textup{[10]} & $D_4$ & $D_4$ & $v\oplus z^2$ & & \parbox{35pt}{\includegraphics[scale=0.2]{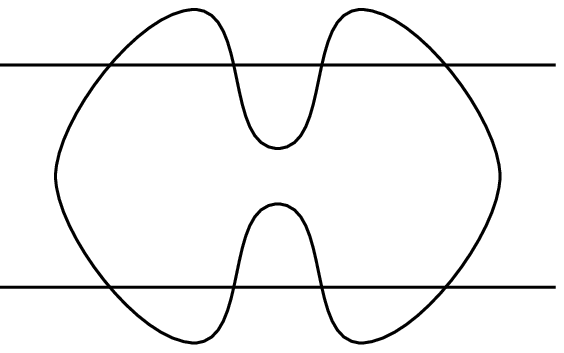}}\\ 
\textup{[11]} & $D_4$ & $D_4$ & $v\oplus z^2$ & $w\oplus z^2$ & \parbox{35pt}{\includegraphics[scale=0.2]{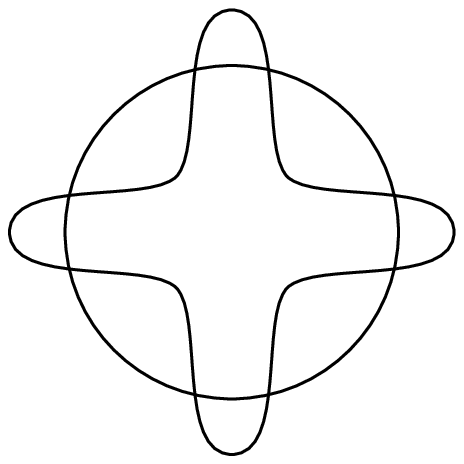}}\\ 
\textup{[12]} & $D_4$ & $D_4$ & $w\oplus z^2$ & & \parbox{35pt}{\includegraphics[scale=0.2]{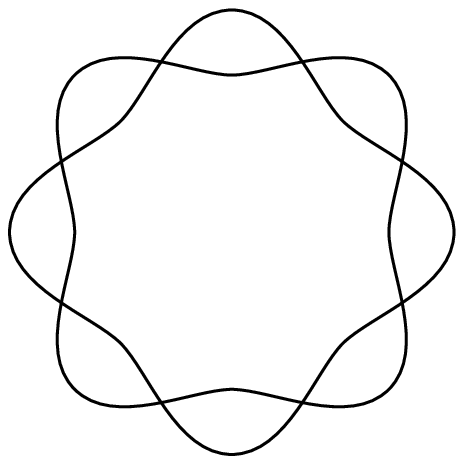}}\\ 
\textup{[13]} & $D_4$ & $D_4$ & $w\oplus z^2$ & $z^2$ & $($See Subsection~$\ref{sextic})$ \\ \midrule
\textup{[14]} & $A_1^4$ & $A_1^4$ & $w^4$ & & \parbox{35pt}{\includegraphics[scale=0.2]{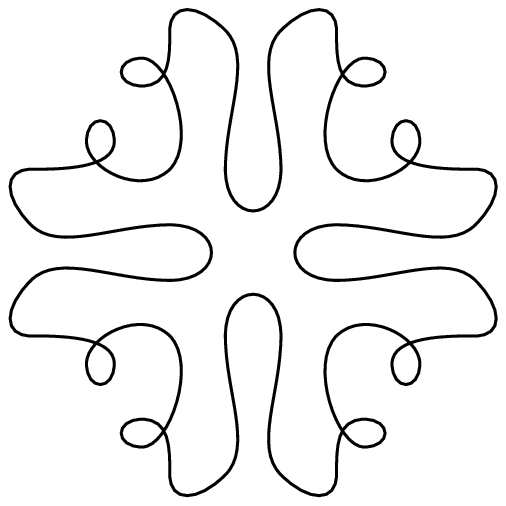}}\\ \midrule
\textup{[15]} & $D_4\oplus A_1$ & $A_1^3$ & $w^3$ & & \parbox{35pt}{\includegraphics[scale=0.2]{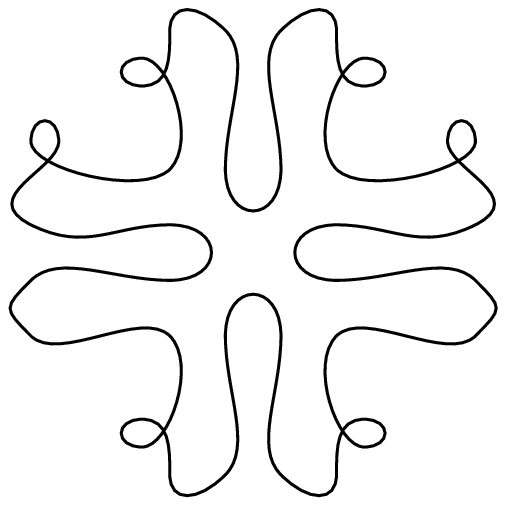}}\\ \midrule
\textup{[16]} & $D_6$ & $A_1^2$ & $w^2$ & & \parbox{35pt}{\includegraphics[scale=0.2]{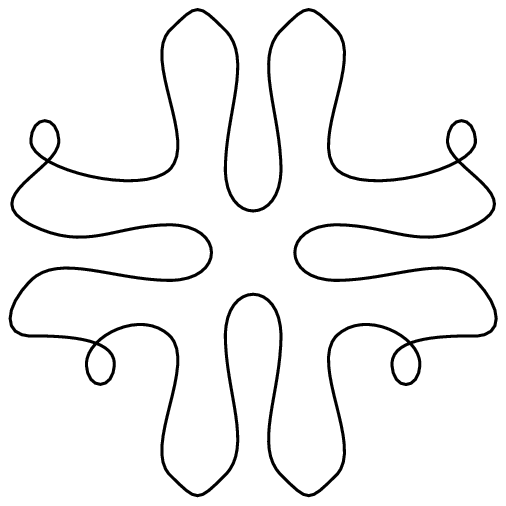}}\\ \midrule
\textup{[17]} & $E_7$ & $A_1$ & $w$ & & \parbox{35pt}{\includegraphics[scale=0.2]{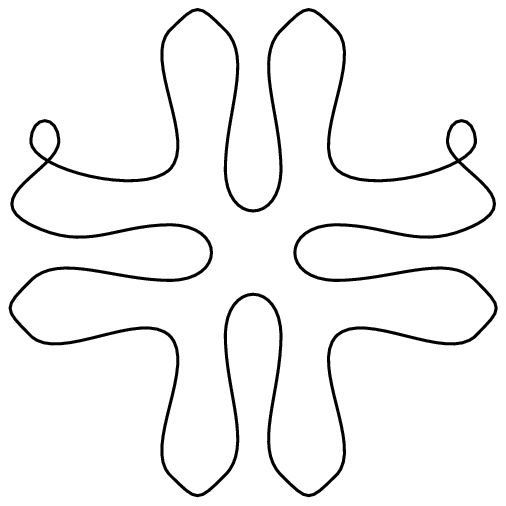}}\\ \midrule
\textup{[18]} & $E_8$ & $\{0\}$ & --- & & \parbox{35pt}{\includegraphics[scale=0.2]{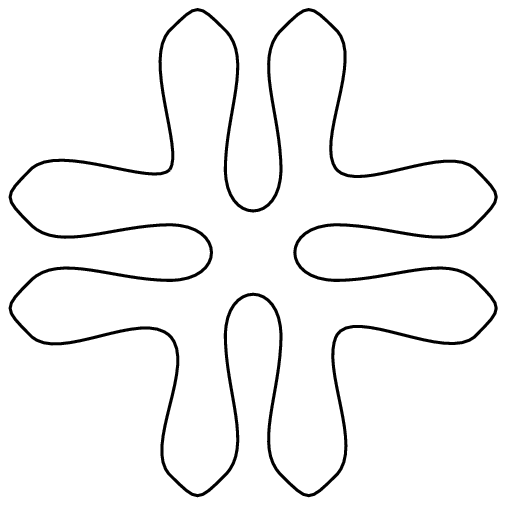}}\\ \bottomrule
\end{longtable}
In Table \ref{table1}, the blank in $q_{S_-}|_{\widetilde{H_-}}$ stands for the same as $q_{S_-}|_{H_-}$. 
Further invariants are collected in the next table. 
\setlongtables
\begin{longtable}{cllcl}
\caption{Further Invariants}\label{table2} \\ \toprule
\text{No.} & $k_-$ & $K_-$ & $(r,l,\delta )$ & Fixed curves \\ \toprule 
\endfirsthead
\toprule
\text{No.} & $k_-$ & $K_-$ & $(r,l,\delta )$ & Fixed curves \\ \toprule  \endhead \endfoot
\textup{[1]} & $u$ & $U\oplus U(2)$ & $(18,2,0)$ & $C^{(1)}+4\mathbb{P}^1$ \\ 
\textup{[2]} & $u^2$ & $U(2)\oplus U(2)$ & $(18,4,0)$ & $4\mathbb{P}^1$ \\ 
\textup{[3]} & $u^2$ & $U(2)\oplus U(2)$ & $(18,4,0)$ & $4\mathbb{P}^1$ \\ \midrule
\textup{[4]} & $u\oplus \langle\frac{-1}{4}\rangle$ & $U\oplus U(2) \oplus A_1(2)$ & $(16,4,1)$ & $C^{(1)}+3\mathbb{P}^1$ \\ 
\textup{[5]} & $u^2\oplus \langle\frac{-1}{4}\rangle$ & $U(2)\oplus U(2) \oplus A_1(2)$ & $(16,6,1)$ & $3\mathbb{P}^1$ \\ \midrule
\textup{[6]} & $u\oplus \langle\frac{-1}{4}\rangle^2$ & $U\oplus U(2) \oplus A_1(2)^2$ & $(14,6,1)$ & $C^{(1)}+2\mathbb{P}^1$ \\ 
\textup{[7]} & $u^2\oplus \langle\frac{-1}{4}\rangle^2$ & $U(2)\oplus U(2) \oplus A_1(2)^2$ & $(14,8,1)$ & $2\mathbb{P}^1$ \\ \midrule
\textup{[8]} & $u\oplus \langle\frac{-1}{4}\rangle^3$ & $U\oplus U(2) \oplus A_1(2)^3$ & $(12,8,1)$ & $C^{(1)}+\mathbb{P}^1$ \\ 
\textup{[9]} & $u^2\oplus \langle\frac{-1}{4}\rangle^3$ & $U(2)\oplus U(2) \oplus A_1(2)^3$ & $(12,10,1)$ & $\mathbb{P}^1$ \\ \midrule
\textup{[10]} & $u\oplus v\oplus v(4)$ & $U\oplus U(2) \oplus D_4(2)$ & $(10,6,0)$ & $C^{(2)}+\mathbb{P}^1$ \\ 
\textup{[11]} & $u\oplus v\oplus v(4)$ & $U\oplus U(2) \oplus D_4(2)$ & $(10,8,0)$ & $C^{(1)}_1+C_2^{(1)}$ \\ 
\textup{[12]} & $u^2\oplus v\oplus v(4)$ & $U(2)\oplus U(2) \oplus D_4(2)$ & $(10,8,0)$ & $C^{(1)}$ \\ 
\textup{[13]} & $u^2\oplus v\oplus v(4)$ & $U(2)\oplus U(2) \oplus D_4(2)$ & $(10,10,0)$ & $\emptyset$ \\ \midrule
\textup{[14]} & $u\oplus \langle\frac{1}{4}\rangle^4$ & $U\oplus U(2) \oplus A_1(2)^4$ & $(10,10,1)$ & $C^{(1)}$ \\ \midrule
\textup{[15]} & $u^2\oplus \langle\frac{1}{4}\rangle^3$ & $U\oplus U(2)\oplus D_4(2)\oplus A_1(2)$ & $(8,8,1)$ & $C^{(2)}$ \\ \midrule
\textup{[16]} & $u^3\oplus \langle\frac{1}{4}\rangle^2$ & $U\oplus U(2) \oplus D_6(2)$ & $(6,6,1)$ & $C^{(3)}$ \\ \midrule
\textup{[17]} & $u^4\oplus \langle\frac{1}{4}\rangle$ & $U\oplus U(2) \oplus E_7(2)$ & $(4,4,1)$ & $C^{(4)}$ \\ \midrule
\textup{[18]} & $u^5$ & $U\oplus U(2) \oplus E_8(2)$ & $(2,2,0)$ & $C^{(5)}$ \\ \bottomrule
\end{longtable}

In Table~$\ref{table2}$, $k_-$ is the invariant defined in Section~$4$, \eqref{k_{pm}} and $(r,l,\delta)$ is the main invariant of the non-symplectic involution $\theta=g\circ \varepsilon$, Section~$\ref{examples}$. 
``Fixed curves'' stands for the $1$-dimensional components of the fixed locus of $\iota$ on $Y$. 
We also note that $K_-$ corresponds generically to the transcendental lattice of the covering $K3$ surface $X$.
\end{theorem}

The Enriques surface of type~[1] was constructed by Horikawa~\cite{Hor}, and studied by Dolgachev~\cite{Dol} and Barth-Peters~\cite{BP}. 
Type~[2] was found by Kondo~\cite{Kon} and constructed generally by Mukai~\cite{Muk1}. 
Type~[3] was constructed by Lieberman (cf.~\cite{MN}). 
The Enriques surfaces of type~[1]--[3] were studied by Mukai-Namikawa~\cite{MN} and Mukai~\cite{Muk1} as numerically trivial involutions. 
Moreover type~[5] was studied by Mukai~\cite{Muk2} as numerically reflective involutions. 

In Section~2 we collect some basic definitions and notation of lattice theory. 
In Section~3 we show that Nikulin's classification theory \cite{Nik4} is useful for our purpose
and we introduce this theory in Section~4.
In Section~5 we classify the lattice structures of involutions into 18 types of the tables in Theorem~\ref{MainThm}. 
We determine the lattices $S_{\pm}, K_{-}$ and forms $q_{S_-}|_{H_-}, q_{S_-}|_{\widetilde{H_-}}, k_-$ here.
In Section~6 we determine the other invariants, give the examples and complete Theorem~\ref{MainThm}. 

The authors wish to express their gratitude to Professor Kondo for suggestions to this problem and 
many stimulating conversations. 
We are grateful to Professor Tokunaga for the construction of the curve in Example~No.~[14].

\section{Preliminaries}

Our main tool is the lattice theory. 
Here we recall some definitions and notations. 

A \emph{lattice} is a pair $(L, (\ , \ ))$, where $L$ is a free $\mathbb{Z}$-module of finite rank and $(\ , \ )$ is a non-degenerate integral symmetric bilinear form on $L$. 
We abbreviate $(L, (\ , \ ))$ to $L$. 
We will denote by $L(m)$ the lattice $(L, m(\ , \ ))$ for a given lattice $(L, (\ , \ ))$ and $m\in \mathbb{Q}$. 
$L$ is called \emph{even} if $(x, x)\in 2\mathbb{Z}$ for all $x\in L$. 
For a lattice $L$, there exists an injective homomorphism $\alpha\colon L\to L^*=\Hom(L, \mathbb{Z})$ defined by $x\mapsto (x, -)$. 
$L$ is called \emph{unimodular} if $\alpha$ is bijective. 
Let $U$ (resp.~$\langle n\rangle$) denote the rank $2$ (resp.~rank $1$) lattice given by the matrix 
\[\begin{pmatrix}0 & 1\\ 1 & 0\end{pmatrix}\quad \text{(resp.~$\begin{pmatrix}n\end{pmatrix}$)}. \]
The root lattices $A_l,\ D_m,\ E_n$ are considered to be negative definite. 

A \emph{finite quadratic form} is a triple $(A, b, q)$, where $A$ is a finite abelian group, 
$b\colon A\times A\to \mathbb{Q}/\mathbb{Z}$ is a symmetric bilinear form, 
and $q$ is a map $q\colon A \to \mathbb{Q}/2\mathbb{Z}$ satisfying the following conditions: 
\begin{enumerate}
\item $q(na)=n^2 q(a)$ for all $n\in \mathbb{Z}$, $a\in A$. 
\item $q(a+a')\equiv q(a)+q(a')+2b(a, a') \pmod{2}$ for all $a, a'\in A$. 
\end{enumerate}
A finite quadratic form is called \emph{non-degenerate} if $b$ is non-degenerate. 
An element $x\in A$ is called \emph{characteristic} if $b(x, a)\equiv q(a)\pmod{1}$ for all $a\in A$. 
We abbriviate $(A, b, q)$ (resp.~$b(a, a')$, $q(a)$) to $(A, q_A)$ or just $q_A$ (resp.~$aa'$, $a^2$). 
We denote by $w$ (resp.~$z$) the finite quadratic form on $\mathbb{Z}/2\mathbb{Z}$ whose value is $1$ (resp.~$0$). 
Note that $w$ and $z$ are degenerate. 

A \emph{discriminant $($quadratic$)$ form} for an even lattice $L$ is a non-degenerate finite quadratic form $(A_L, b_L, q_L)$, 
where $A_L:=L^*/L$, $b_L(\bar{x}, \bar{y})=(x, y) \pmod{\mathbb{Z}}$, and $q_L(\bar{x})=(x, x) \pmod{2\mathbb{Z}}$. 
We denote by $u$ (resp.~$v$, $\langle\frac{1}{n}\rangle$) the associated discriminant form of the lattice $U(2)$ (resp.~$D_4$, $\langle n\rangle$). 
We often use the following discriminant forms: 
\begin{align*}
(L, q_L)=\; & (A_1(2), \langle \tfrac{-1}{4}\rangle),\ (D_4(2), v\oplus v(4)), \\ 
& (D_6(2), u^2\oplus \langle\tfrac{1}{4}\rangle^2),\ (E_7(2), u^3\oplus \langle\tfrac{1}{4}\rangle),\ (E_8(2), u^4), 
\end{align*}
where $u^n$ denotes $n$ copies of $u$ and $v(4)$ denotes 
\[((\mathbb{Z}/4\mathbb{Z})^2, \begin{pmatrix}\frac{1}{2} & \frac{1}{4} \\ \frac{1}{4} & \frac{1}{2}\end{pmatrix}). \]

An embedding $i\colon S\to L$ of lattices is called \emph{primitive} if $L/i(S)$ is free. 
Let $S$ be a sublattice of $L$. 
We define the sublattices 
\begin{align*}
S^{\perp}&:=\{x\in L \mid (x, y)=0\quad \forall y\in S\}, \\
S^{\wedge}&:=S\otimes \mathbb{Q}\cap L 
\end{align*}
of $L$ called the \emph{orthogonal complement} to $S$ and the \emph{primitive closure} of $S$ respectively. 
Let $T$ be an orthogonal sublattice to $S$. 
We write 
\[\Gamma_{ST}:=(S\oplus T)^{\wedge}/(S\oplus T). \]
Two primitive embeddings $i\colon S\to L$ and $i'\colon S\to L'$ are called \emph{isomorphic} if there exists $f\in \Isom(L, L')$ such that $f\circ i=i'$. 

Let $M$ and $N$ be even lattices, and let $M\to N$ be an embedding. 
Then $N$ is called an \emph{overlattice} of $M$ if $N/M$ is a finite abelian group. 
Let $l(A)$ denote the minimal number of generators of an abelian group $A$. 
Note that 
\begin{equation}\label{rank>=l}
\rank M \geqq l(A_M), \quad l(A_N) \geqq l(A_M)-2l(N/M)
\end{equation}
for a lattice $M$ and an overlattice $N$ of $M$. 

A lattice $M$ is called \emph{$2$-elementary} if $A_M=M^*/M$ is a $2$-elementary group $(\mathbb{Z}/2\mathbb{Z})^a$. 
\begin{proposition}[{\cite[Theorem~3.6.2]{Nik2}}]\label{NikThm362}
The isomorphism class of an even hyperbolic $2$-elementary lattice $M$ is determined by the invariants $(r, l, \delta)$, 
where $r$ is the rank of $M$, $l$ is the minimal number of generators of $A_M$, and $\delta$ is the parity of $q_M$, that is, 
\[\delta=\begin{cases}0 & \text{if $q_M(x)=0\quad \forall x\in A_M$, } \\ 1 & \text{otherwise. }\end{cases}\]
\end{proposition}

Let $L$ be a lattice and $\sigma$ an involution of $L$. 
Write 
\begin{align*}
L^{\langle\sigma\rangle}&=\{x\in L \mid \sigma(x)=x\}, \\
L_{\langle\sigma\rangle}&=(L^{\langle\sigma\rangle})^{\perp}=\{x\in L \mid \sigma(x)=-x\}. 
\end{align*}
Note that if $L$ is unimodular, then $L^{\langle\sigma\rangle}$ and $L_{\langle\sigma\rangle}$ are $2$-elementary lattices. 

Next proposition is the analogue of Witt's theorem. 
\begin{proposition}[{\cite[Prop~1.9.2]{Nik4}}]\label{NikProp192}
Let $q$ be a finite quadratic form on a finite $2$-elementary group $Q$ whose kernel is zero, that is, 
\[\{x\in Q \mid x\perp Q \text{ and } q(x)=0\} = \{0\}. \]
Let $\theta\colon H_1\to H_2$ be an isomorphism of two subgroups of $Q$ that preserves the restrictions $q|H_1$ and $q|H_2$ 
and that maps the elements of the kernel and the characteristic elements of the bilinear form $q$ into the same sort of elements if they belong to $H_1$. 
Then $\theta$ extends to an automorphism of $q$. 
\end{proposition}

%%%%%%%%%%%%%%%%%%%%%%%%%%%%%%%%%%%%%%%%%%%%%%%%%%
\section{Involutions on Enriques surfaces}

Let $Y$ be an Enriques surface and $X$ its covering $K3$ surface with the covering involution $\epsilon$. 
Consider an involution $\iota$ of $Y$. Then $\iota$ lifts to two involutions of $X$. 
One of them acts on $H^0(X, \Omega^2)$ trivially, which we denote by $g$. Then another involution is $g\circ\epsilon=\epsilon\circ g$. 

The second cohomology group $H^2(X, \mathbb{Z})$ is an even unimodular lattice with the signature $(3, 19)$. 
Let $S=\{x\in H^2(X, \mathbb{Z}) \mid g^*(x)=-x\}$, where $g^*$ is the involution of $H^2(X, \mathbb{Z})$ induced by $g$. 
It is known that $S$ is isomorphic to $E_8(2)$ and this does not depend on $g$ (\cite{Mor}, \cite{Nik1}). 

\begin{lemma}\label{2-elm}
Let $L$ be a unimodular lattice and $S$ a $2$-elementary lattice. 
The followings are equivalent. 
\begin{enumerate}
\item There exists an involution $\alpha$ of $L$ such that $L_{\langle\alpha\rangle}\cong S$. 
\item There exists a primitive embedding $S\to L$. 
\end{enumerate}
\end{lemma}

\begin{proof}
Assume (1). 
Since $S=(L^{\langle\alpha\rangle})^{\perp}$, it follows that the sublattice $S$ is primitive in $L$. 

Assume (2). 
Let $K=S^{\perp}$. 
Since $S$ and $K$ are $2$-elementary lattices, 
there exists an involution $\alpha\in\rmO(L)$ such that $\alpha|_K=1$ and $\alpha|_S=-1$, by \cite[Corollary~1.5.2]{Nik2}. 
Since $S$ is primitive in $L$, it follows that $S=L_{\langle\alpha\rangle}$. 
\end{proof}

To classify $\iota$, it suffices to classify the pair of involutions $(g, \epsilon)$. 
From Torelli type theorem (\cite{PS}), this is equivalent to classifying the pair $(g^*, \epsilon^*)$. 
By Lemma~\ref{2-elm}, this is equivalent to classifying a primitive embedding of $S=E_8(2)$ into $H^2(X, \mathbb{Z})$ and an action of $\epsilon^*$ on $S$. 

%%%%%%%%%%%%%%%%%%%%%%%%%%%%%%%%%%%%%%%%%%%%%%%%%%
\section{Involutions of a lattice with condition on a sublattice}

In this section, we introduce the theory of involutions of a lattice with condition on a sublattice. 

\begin{definition}[{\cite[Definition~1.1.1]{Nik4}}]
By a \emph{condition on an involution} we understand a pair $(S, \theta)$, 
where $S$ is a non-degenerate lattice and $\theta$ is an involution of $S$. 
\end{definition}

\begin{remark}
In \cite{Nik4}, a condition on an involution is defined as a triple $(S, \theta, G)$, 
where $S$ is a (possibly degenerate) lattice, $\theta$ is an involution of $S$, 
and $G\subset \rmO(S, \theta)$ is a distinguished subgroup of the normalizer of $\theta$ in $\rmO(S)$. 
In this paper, we assume that $G=\{\id_S\}$. 
\end{remark}

\begin{definition}[{\cite[Definition~1.1.2]{Nik4}}]
By a \emph{unimodular involution with the condition $(S, \theta)$} we understand a triple $(L, \phi, i)$, 
where $L$ is a unimodular lattice, $\phi$ is an involution of $L$, 
and $i\colon S\to L$ is a primitive embedding satisfying $\phi\circ i = i\circ \theta$. 

Two unimodular involutions $(L, \phi, i)$ and $(L', \phi', i')$ with the condition $(S, \theta)$ are called \emph{isomorphic} 
if there exists an isomorphism $f\colon L\to L'$ with $\phi'\circ f=f\circ \phi$ and $f\circ i=i'$. 
\end{definition}

Let $S_{\pm}=\{x\in S \mid \theta(x)=\pm x\}$. 
We write $p_{S_{\pm}}\colon S/(S_+\oplus S_-) \to A_{S_{\pm}}$ for the projections and $\Gamma_{\pm}=p_{S_{\pm}}(S/(S_+\oplus S_-))\subset A_{S_{\pm}}$ for the images of $S/(S_+\oplus S_-)$. 
Note that $S/(S_+\oplus S_-)$ is the graph of $\gamma:=p_{S_-}\circ p_{S_+}^{-1}\colon \Gamma_+\to \Gamma_-$, so we write $\Gamma_{\gamma}=S/(S_+\oplus S_-)$. 

\begin{theorem}[{\cite[Theorem~1.3.1]{Nik4}}]\label{NikThm131}
Any unimodular involution with the condition $(S, \theta)$ is determined by the list 
\begin{equation}\label{list}
(H_{\pm}, q_r, q, \gamma_r, K_{\pm}, \gamma_{K_{\pm}}), 
\end{equation}
where $H_{\pm}$ are subgroups with $\Gamma_{\pm} \subset H_{\pm} \subset (S_{\pm}^*\cap \tfrac{1}{2}S_{\pm})/S_{\pm}$, 
$q_r$ is a finite quadratic form on the $2$-elementary group $(H_+\oplus H_-)/\Gamma_{\gamma}$ with $q_r|_{H_{\pm}}=\pm q_{S_{\pm}}|_{H_{\pm}}$, 
$q$ is the isomorphism class of a non-degenerate $2$-elementary finite quadratic form, 
$\gamma_r\colon q_r \to q$ is an embedding of forms, $K_{\pm}$ are even lattices, 
and $\gamma_{K_{\pm}}\colon q_{K_{\pm}} \to k_{\pm}$ are isomorphisms of forms. 
Here $k_{\pm}$ are defined by 
\begin{equation}\label{k_{pm}}
k_{\pm}=((-q_{S_{\pm}}\oplus \pm q)|\Gamma_{\gamma_{r}|H_{\pm}}^{\perp})/\Gamma_{\gamma_{r}|H_{\pm}}, 
\end{equation}
where $\Gamma_{\gamma_{r}|H_{\pm}}$ are the graphs of the embeddings $H_{\pm} \to q$ induced by $\gamma_r$. 

Two lists $(\ref{list})$ and $(H_{\pm}', q_r', q', \gamma_r', K_{\pm}', \gamma_{K_{\pm}'}')$ 
determine isomorphic unimodular involutions with the condition $(S, \theta)$ if and only if $H_{\pm}=H_{\pm}'$, $q_r=q_r'$, $q=q'$, 
and there exist isomorphisms $\xi \in \rmO(q)$ and $\psi_{\pm} \in \Isom(K_{\pm}, K_{\pm}')$ 
such that $\xi\circ\gamma_r=\gamma_r'$ and $(\id, \xi)|_{k_{\pm}}\circ\gamma_{K_{\pm}}=\gamma_{K_{\pm}'}'\circ\overline{\psi_{\pm}}$, 
where $(\id, \xi)|_{k_{\pm}}$ are isomorphisms between $k_{\pm}$ and $k_{\pm}'$ induced by $\id\in \rmO(q_{S_{\pm}})$ and $\xi$, 
and $\overline{\psi_{\pm}}$ are isomorphisms between $q_{K_{\pm}}$ and $q_{K_{\pm}'}$ induced by $\psi_{\pm}$. 
\end{theorem}

\begin{proof}
We prove only the assertion about the equivalence of the lists (\ref{list}), which is omitted in \cite{Nik4}. 
Let $(L, \phi, i)$ and $(L', \phi', i')$ be the unimodular involutions with the condition $(S, \theta)$ 
determined by the lists $(\ref{list})$ and $(H_{\pm}', q_r', q', \gamma_r', K_{\pm}', \gamma_{K_{\pm}'}')$ respectively. 

Assume that two lists determine isomorphic unimodular involutions. 
There exists $f\in \Isom(L, L')$ such that $f\circ i=i'$ and $\phi'\circ f=f\circ\phi$. 
It follows from $\phi'\circ f=f\circ\phi$ that $f$ induces $f_{\pm}:=f|_{L_{\pm}}\in \Isom(L_{\pm}, L_{\pm}')$ with $f_{\pm}\circ i_{\pm}=i'_{\pm}$, 
where $i_{\pm}\colon S_{\pm}\to L_{\pm}$ and $i_{\pm}'\colon S_{\pm}\to L_{\pm}'$ are primitive embeddings induced by $i$ and $i'$ respectively. 
Since $f$ induces $f|_{L_+\oplus S_-}=(f_+, \id)\in \Isom(L_+, L_+')\times \rmO(S_-)$, so does an isomorphism between $(L_+\oplus S_-)^{\wedge}$ and $(L_+'\oplus S_-)^{\wedge}$. 
Hence we have 
\[H_-=p_{S_-}(\Gamma_{L_+S_-})=p_{S_-}(\Gamma_{L_+'S_-})=H_-'\]
and $\overline{f_+}\circ \gamma_r|_{H_-}=\gamma_r'|_{H_-'}$, 
where $\overline{f_+}$ is an isomorphism between $q$ and $q'$ induced by $f_+$. 

Similarly $f$ induces an isomorphism between $(L_-\oplus S_+)^{\wedge}$ and $(L_-'\oplus S_+)^{\wedge}$. 
Hence we see that $H_+=H_+'$ and $\overline{f_-}\circ (\gamma_{L_+L_-}\circ \gamma_r|_{H_+})=\gamma_{L_+'L_-'}\circ \gamma_r'|_{H_+'}$. 
From $\overline{f_-}\circ\gamma_{L_+L_-}=\gamma_{L_+'L_-'}\circ \overline{f_+}$, we have $\overline{f_+}\circ \gamma_r=\gamma_r'$. 
Since $(L_+\oplus S_-)^{\wedge}=(K_-)_L^{\perp}$ and $(L_+'\oplus S_-)^{\wedge}=(K_-')_{L'}^{\perp}$, there exists $\psi_-$ with the condition, by \cite[Corollary~1.5.2]{Nik2}. 
Similarly, we have $\psi_+$ with the condition. 
It is clear that $q=q'$ and $q_r=q_r'$. 

We turn to the contrary. 
Assume that $H_{\pm}=H_{\pm}'$, $q_r=q_r'$, $q=q'$ and there exist $\xi=\xi_+\in \rmO(q)$ and $\psi_{\pm} \in \Isom(K_{\pm}, K_{\pm}')$ with the conditions. 
Note that invariants $(H_{\pm}, \gamma_r|_{H_{\pm}}, K_{\pm}, \gamma_{K_{\pm}})$ determine primitive embeddings $i_{\pm}\colon S_{\pm}\to L_{\pm}$ 
with orthogonal complements $K_{\pm}$ by \cite[Proposition~1.15.1]{Nik2}, 
where $L_{\pm}$ are the lattices with discriminant forms $\pm q$ respectively. 
Let $T_1$ (resp.~$T_2$) be any lattice which is the unique in its genus 
and furthermore $\rmO(T_1)\to \rmO(q_{T_1})$ (resp.~$\rmO(T_2)\to \rmO(q_{T_2})$) is surjective and $q_{T_1}=q$ (resp.~$q_{T_2}=-q$). 
From $q=q'$ and $K_-\cong K_-'$ (resp.~$K_+\cong K_+'$), 
we see that $L_-$ and $L_-'$ (resp.~$L_+$ and $L_+'$) are obtained as orthogonal complements of a primitive embedding $T_1\to L_1$ (resp.~$T_2\to L_2$), 
where $L_1$ (resp.~$L_2$) is a unimodular lattice with 
\begin{gather*}
\Sign L_1=\Sign L_-+\Sign T_1=\Sign L_-'+\Sign T_1 \\
\text{(resp.~$\Sign L_2=\Sign L_++\Sign T_2=\Sign L_+'+\Sign T_2$). }
\end{gather*}
Moreover $T_1$ is obtained as an orthogonal complement of a primitive embedding $T_2\to L_3$, 
where $L_3$ is a unimodular lattice with 
\[\Sign L_3=\Sign T_1+\Sign T_2. \]
Hence there exists $\xi_-\in \rmO(-q)$ such that $\xi_-\circ \gamma_{T_1T_2}=\gamma_{T_1T_2}\circ \xi_+$. 

Since $\rmO(T_1)\to \rmO(q_{T_1})=\rmO(q)$ is surjective, there exists $f_1\in \rmO(T_1)$ such that $\overline{f_1}=\xi_+$. 
By $\xi_+\circ \gamma_r|_{H_-}=\gamma_r'|_{H_-'}$ and $H_-=H_-'$, it follows that $(f_1, \id)\in \rmO(T_1)\times \rmO(S_-)$ extends to 
an isomorphism 
\[\alpha_1\colon (T_1\oplus S_-)^{\wedge}\to (T_1\oplus S_-)^{\wedge}. \]
Note that the former $(T_1\oplus S_-)^{\wedge}$ is equal to $(K_-)^{\perp}_{L_1}$, and the latter is equal to $(K_-')^{\perp}_{L_1}$. 
From the condition of $\psi_-$, it follows that $(\alpha_1, \psi_-)$ extends to an automorphism 
\[\beta_1\colon L_1\to L_1. \]
Similarly there exists an automorphism $\beta_2\colon L_2\to L_2$ such that $\overline{\beta_2|_{T_2}}\in \rmO(-q)$, $\beta_2|_{S_+}=\id$ and $\beta_2|_{K_+}=\psi_+$. 
Therefore we have the following commutative diagram: 
\[\xymatrix{
A_{L_-} \ar[r] \ar[d]_{\overline{\beta_1|_{L_-}}} & A_{T_1} \ar[r] \ar[d]^{\xi_+} & A_{T_2} \ar[r] \ar[d]_{\xi_-} & A_{L_+} \ar[d]^{\overline{\beta_2|_{L_+}}} \\
A_{L_-'} \ar[r] & A_{T_1} \ar[r] & A_{T_2} \ar[r] & A_{L_+'} 
}\]
Hence $(\beta_2|_{L_+}, \beta_1|_{L_-})$ extends to an isomorphism $\beta\colon L\to L'$ with $\beta\circ i=i'$ and $\beta\circ\phi=\phi'\circ \beta$, 
which is the desired isomorphism. 
\end{proof}

\begin{remark}\label{beta}
In the proof of Theorem~\ref{NikThm131}, we see that 
\[\overline{\beta_2|_{L_+}}=(\overline{\psi_+}, \overline{\id})|\Gamma_{K_+S_+}^{\perp}/\Gamma_{K_+S_+}, \quad 
\overline{\beta_1|_{L_-}}=(\overline{\psi_-}, \overline{\id})|\Gamma_{K_-S_-}^{\perp}/\Gamma_{K_-S_-}. \]
Moreover, if $L_{\pm}$ is indefinite, then we can take $T_1$ and $T_2$ as $L_+$ and $L_-$ respectively. 
Hence we see that 
\begin{equation}\label{xi_{pm}}
\xi_+=(\overline{\psi_+}, \overline{\id})|\Gamma_{K_+S_+}^{\perp}/\Gamma_{K_+S_+}, \quad 
\xi_-=(\overline{\psi_-}, \overline{\id})|\Gamma_{K_-S_-}^{\perp}/\Gamma_{K_-S_-}. 
\end{equation}
\end{remark}

%%%%%%%%%%%%%%%%%%%%%%%%%%%%%%%%%%%%%%%%%%%%%%%%%%
\section{Classification}

The construction of the list (\ref{list}) from the unimodular involution with condition is as follows (see \cite{Nik4} for more details): 
Let $(L, \phi, i)$ be a unimodular involution with the condition $(S, \theta)$. 
We write 
\[L_{\pm}=\{x\in L \mid \phi(x) = \pm x\}. \]
Define $q:=q_{L_+}$. 
The primitive embedding $i\colon S\to L$ defines primitive embeddings $i_{\pm}\colon S_{\pm}\to L_{\pm}$. 
Hence we define $2$-elementary groups 
\[H_{\pm}:=p_{S_{\pm}}(\Gamma_{L_{\mp}S_{\pm}})\subset (S_{\pm}^*\cap \tfrac{1}{2}S_{\pm})/S_{\pm}. \]
Note that $p_{S_{\pm}}$ are injective, since $i_{\pm}$ are primitive. 
We can also say that $\Gamma_{L_-S_+}$ (resp.~$\Gamma_{L_+S_-}$) is the graph of injective homomorphism 
\[\gamma_{H_+}\colon H_+\to A_{L_-} \quad \text{(resp.~$\gamma_{H_-}\colon H_-\to A_{L_+}$)}. \]
Note that the notation of $\gamma_{H_{\pm}}$ is slightly different from that of \cite{Nik4}. 
We define the embedding of forms $\gamma_r$ and the quadratic form $q_r$ on $(H_+\oplus H_-)/\Gamma_{\gamma}$ as 
\begin{align*}
\gamma_r&:=(\gamma_{L_+L_-}^{-1}\circ \gamma_{H_+}, \gamma_{H_-})\colon H_+\oplus H_-/\Gamma_{\gamma} \to q, \\
q_r&:=q\circ \gamma_r,  
\end{align*}
where $\gamma_{L_+L_-}$ is an isomorphism between $A_{L_+}$ and $A_{L_-}$. 
The even lattices $K_{\pm}$ are defined by $K_{\pm}:=(S_{\pm})^{\perp}_{L_{\pm}}$. 
The quadratic forms $-k_{\pm}$ in (\ref{k_{pm}}) are equal to discriminant forms of $(L_{\mp}\oplus S_{\pm})^{\wedge}$. 
Hence the sign reversing isometies give $\gamma_{K_{\pm}}\colon q_{K_{\pm}}\to k_{\pm}$. 

From now on, we regard $L=H^2(X, \mathbb{Z})$, $\phi=\epsilon^*$, and $S=\{x\in L \mid g^*(x)=-x\}\cong E_8(2)$. 
It is known that 
\[L_+\cong U(2)\oplus E_8(2),\quad L_-\cong U\oplus U(2)\oplus E_8(2)\]
and these do not depend on $\epsilon$ (\cite{BP}). 

\begin{lemma}\label{S_{pm}}
Suppose that $S=E_8(2)$ and $\theta$ is an involution of $S$. 
Then the isomorphism class of $(S_+, S_-)$ is one of the following: 
\begin{align*}
(S_+(\tfrac{1}{2}), S_-(\tfrac{1}{2}))=&\ (E_8, \{0\}),\ (E_7, A_1),\ (D_6, A_1^2),\ (D_4\oplus A_1, A_1^3), (D_4, D_4), \\
&\ (A_1^4, A_1^4),\ (A_1^3, D_4\oplus A_1),\ (A_1^2, D_6),\ (A_1, E_7),\ (\{0\}, E_8). 
\end{align*} 
%\[\begin{array}{c|cccccccccc}
%S_+(\frac{1}{2}) & E_8 & E_7 & D_6 & D_4\oplus A_1 & D_4 & A_1^4 & A_1^3 & A_1^2 & A_1 & \{0\} \\ \hline
%S_-(\frac{1}{2}) & \{0\} & A_1 & A_1^2 & A_1^3 & D_4 & A_1^4 & D_4\oplus A_1 & D_6 & E_7 & E_8
%\end{array}. \]
%Note that determining $(S_+, S_-)$ is equivalent to determining $\theta$. 
\end{lemma}

\begin{proof}
It suffices to prove the lemma for $S(\frac{1}{2})=E_8$. 
Since $\theta$ is an involution, it follows that $S_{\pm}$ are even $2$-elementary lattices. 
We can assume that the rank of $S_+$ is at most $4$. 
By \cite[Theorem~3.6.2]{Nik2}, invariants $(r, l, \delta)$ of $S_+$ is one of the following: 
\[(0, 0, 0),\ (1, 1, 1),\ (2, 2, 1),\ (3, 3, 1),\ (4, 4, 1),\ (4, 2, 0). \]
We see that $\{0\}$, $A_1$, $A_1^2$, $A_1^3$, $A_1^4$ and $D_4$ have above invariants respectively, 
and these lattices have one class in their genus (cf.~\cite[Remark~1.14.6]{Nik2}). 
Hence $S_+$ is one of them. 
$S_-$ is obtained as orthogonal complement to $S_+$. 
\end{proof}

From this lemma, we calculate the list~(\ref{list}) for each $(S_+, S_-)$. 

\begin{lemma}\label{unique-emb}
Suppose that $S_+$ is one of them in Lemma~$\ref{S_{pm}}$. 
Then there exists a unique primitive embedding $S_+\to L_+$. 
\end{lemma}

\begin{proof}
Since $S_+(\frac{1}{2})$ is an even negative definite lattice of rank less than $8$ and $L_+(\frac{1}{2})\cong U\oplus E_8$ is a unimodular lattice of signature $(1, 9)$, 
the lemma follows from \cite[Theorem~1.14.4]{Nik2}. 
\end{proof}

\begin{corollary}\label{K_+=U(2)+S_-}
We have $K_+\cong U(2)\oplus S_-$ in all cases. 
In particular, $\Gamma_{S_+S_-}\cong \Gamma_{K_+S_+}$. 
\end{corollary}

\begin{proof}
Recall that $L_+\cong U(2)\oplus E_8(2)$ and $S\cong E_8(2)$. 
By Lemma~\ref{unique-emb}, a primitive embedding $S_+\to L_+$ is unique. 
Hence $K_+=(S_+)_{L_+}^{\perp}$ is uniquely determined as $U(2)\oplus S_-$. 
Therefore we see that 
$\Gamma_{K_+S_+}=L_+/(K_+\oplus S_+)\cong (U(2)\oplus S)/(U(2)\oplus S_-\oplus S_+)\cong S/(S_+\oplus S_-)=\Gamma_{S_+S_-}$. 
\end{proof}

\begin{lemma}\label{H_{pm}}
On $H_{\pm}$, we have the following: 
\begin{enumerate}
\item $\Gamma_+\subset H_+ = \frac{1}{2}S_+/S_+$, $\Gamma_-\subset H_- \subset \frac{1}{2}S_-/S_-$. 
\item $q_{S_{\pm}}|_{H_{\pm}}\equiv 0 \pmod{1}$. 
\item $\rank S_{-}-1 \leqq \rank H_{-}\leqq \rank S_{-}$. 
\item $q_{S_+}|_{\Gamma_+}$ $($resp.~$q_{S_-}|_{\Gamma_-}$$)$ is a direct summand of $q_{S_+}|_{H_+}$ $($resp.~$q_{S_-}|_{H_-}$$)$. 
\end{enumerate}
\end{lemma}

\begin{proof}
Since $S_{\pm}(\frac{1}{2})$ are even lattices, we have $H_{\pm} \subset \frac{1}{2}S_{\pm}/S_{\pm}$. 
Let $x\in \frac{1}{2}S_+$. 
From $L_+(\frac{1}{2})\cong U\oplus E_8$, we have $x\in L_+^*$. 
Since $L$ is unimodular, there exists $y\in L_-^*$ such that $x+y\in L$, which implies $x+y\in (S_+\oplus L_-)^{\wedge}$. 
Therefore $H_+ = \frac{1}{2}S_+/S_+$. 

Since $\gamma_r\colon q_r\to q$ is an embedding and $q=u^5\equiv 0 \pmod{1}$, $q_r$ also satisfies $q_r\equiv 0 \pmod{1}$. 
Hence $q_{S_{\pm}}|_{H_{\pm}}=q_r|_{H_{\pm}}\equiv 0 \pmod{1}$. 

By $K_-=(S_-)_{L_-}^{\perp}$, we see that $\rank K_- = \rank L_- - \rank S_- = 12-\rank S_-$. 
From (\ref{rank>=l}), we see that 
\[l(A_{K_-})=l(A_{(L_+\oplus S_-)^{\wedge}})\geqq l(A_{L_+\oplus S_-})-2l(\Gamma_{L_+S_-})=10+l(A_{S_-})-2l(\Gamma_{L_+S_-}). \]
Obviously $l(A_{S_-})=\rank S_-$. 
The primitivity of $L_+$ in $(L_+\oplus S_-)^{\wedge}$ gives $l(\Gamma_{L_+S_-})=l(H_-)=\rank H_-$. 
Therefore $\rank K_- \geqq l(A_{K_-})$ yields 
\[12-\rank S_- \geqq 10+\rank S_- - 2\rank H_-. \]
Hence we have (3). 

In our $S_+$, we can write $A_{S_+}=(\mathbb{Z}/2\mathbb{Z})^a\oplus(\mathbb{Z}/4\mathbb{Z})^b$ 
and $q_{S_+}=q_2\oplus q_4$ where $q_2$ (resp.~$q_4$) is a finite quadratic form on $(\mathbb{Z}/2\mathbb{Z})^a$ (resp.~$(\mathbb{Z}/4\mathbb{Z})^b$). 
Since $\Gamma_+=2A_{S_+}=\{2x \mid x\in A_{S_+}\}$, we have $q_{S_+}|_{\Gamma_+}=2q_4$ 
where $2q_4$ denotes a finite quadratic form whose generators are twice the size of those of $q_4$. 
Since $q_{S_+}|_{\frac{1}{2}S_+/S_+}=q_2\oplus 2q_4$, 
we see that $q_{S_+}|_{\Gamma_+}$ is a direct summand of $q_{S_+}|_{\frac{1}{2}S_+/S_+}$. 
Hence $q_{S_+}|_{\Gamma_+}$ is also that of $q_{S_+}|_{H_+}$. 
The same proof works for $q_{S_-}|_{\Gamma_-}$. 
\end{proof}

\begin{lemma}\label{Gamma_{pm}=H_{pm}}
\begin{enumerate}
\item In cases $S_-(\frac{1}{2})=E_8, E_7, D_6, D_4 \oplus A_1$, we have $\Gamma_+=H_+=\tfrac{1}{2}S_+/S_+$. 

\item In cases $S_-(\frac{1}{2})=A_1^4, A_1^3, A_1^2, A_1, \{0\}$, we have $\Gamma_-= H_-=\tfrac{1}{2}S_-/S_-$. 
\end{enumerate}
\end{lemma}

\begin{proof}
We give the proof only for the case $S_-(\frac{1}{2})=E_7$; 
the other cases are left to the reader. 
In case $S_-(\frac{1}{2})=E_7$, we have $S_+(\frac{1}{2})=A_1$. 
Hence we see that 
\begin{align*}
\Gamma_+&=p_{S_+}(S/(S_+\oplus S_-))=p_{S_+}(E_8(2)/(A_1(2)\oplus E_7(2))) \\
&\cong p_{S_+}(E_8/(A_1\oplus E_7))=A_{A_1}\cong \mathbb{Z}/2\mathbb{Z}. 
\end{align*}
At the same time, we see that 
\[\tfrac{1}{2}S_+/S_+=\tfrac{1}{2}A_1(2)/A_1(2)\cong \mathbb{Z}/2\mathbb{Z}. \] 
The lemma follows from Lemma~\ref{H_{pm}}~(1). 
\end{proof}

We consider the behavior of $\gamma_{H_{\pm}}\colon H_{\pm}\to A_{L_{\mp}}$. 
Note that 
\[\Gamma_{K_{\pm}S_{\pm}}^{\perp}\cap A_{K_{\pm}}
=(\Gamma_{K_{\pm}S_{\pm}}^{\perp}\cap A_{K_{\pm}})/(\Gamma_{K_{\pm}S_{\pm}}\cap A_{K_{\pm}})
\subset \Gamma_{K_{\pm}S_{\pm}}^{\perp}/\Gamma_{K_{\pm}S_{\pm}}
=A_{L_{\pm}}. \]

\begin{definition}
Let $\widetilde{A_{K_+}}:=\Gamma_{K_+S_+}^{\perp}\cap A_{K_+} \subset A_{L_+}$ and $\widetilde{A_{K_-}}:=\Gamma_{K_-S_-}^{\perp}\cap A_{K_-} \subset A_{L_-}$. 
The subgroup $\widetilde{H_-}$ of $H_-$ and $\widetilde{H_+}$ of $H_+$ are defined by 
\[\widetilde{H_-}:=\gamma_{H_-}^{-1}(\widetilde{A_{K_+}}), \quad \widetilde{H_+}:=\gamma_{H_+}^{-1}(\widetilde{A_{K_-}}). \]
\end{definition}

We see that $(\widetilde{H_-}, \gamma_{H_-}|_{\widetilde{H_-}})$ and $(\widetilde{H_+}, \gamma_{H_+}|_{\widetilde{H_+}})$ 
determine $(K_+\oplus S_-)^{\wedge}$ and $(S_+\oplus K_-)^{\wedge}$ respectively, 
since $(H_{\mp}, \gamma_{H_{\mp}})$ determine $(L_{\pm}\oplus S_{\mp})^{\wedge}$. 
It follows from Corollary~\ref{K_+=U(2)+S_-} that $\Gamma_+=p_{S_+}(\Gamma_{K_+S_+})$. 
Therefore we have 
\begin{equation}\label{wt{H_-}}
\Gamma_-\subset \widetilde{H_-}\subset H_-. 
\end{equation} 

From Theorem~\ref{NikThm131}, 
if two unimodular involutions with the condition $(S, \theta)$ determined by the lists (\ref{list}) and $(H_{\pm}', q_r', q', \gamma_r', K_{\pm}', \gamma_{K_{\pm}'}')$ respectively are isomorphic, 
then there exist $\xi_{\pm}\in \rmO(\pm q)$ and $\psi_{\pm} \in \Isom(K_{\pm}, K_{\pm}')$ with the conditions. 
As stated in Remark~\ref{beta}, we have (\ref{xi_{pm}}). 
It follows that 
\[\widetilde{H_-}=\widetilde{H_-'} \quad \text{(resp.~$\widetilde{H_+}=\widetilde{H_+'}$), }\]
since $(\overline{\psi_+}, \overline{\id})|\Gamma_{K_+S_+}^{\perp}/\Gamma_{K_+S_+}$ (resp.~$(\overline{\psi_-}, \overline{\id})|\Gamma_{K_-S_-}^{\perp}/\Gamma_{K_-S_-}$) 
induces an isomorphism between $\widetilde{A_{K_+}}$ and $\widetilde{A_{K_+'}}$ (resp.~$\widetilde{A_{K_-}}$ and $\widetilde{A_{K_-'}}$). 
Hence we define the following equivalence relation: 
\begin{align*}
&\gamma_{H_{\mp}}\sim \gamma_{H_{\mp}'}\\ 
\stackrel{\text{def}}{\Longleftrightarrow}\ 
&\text{there exists $\xi_{\pm} \in \rmO(\pm q)$ such that $\xi_{\pm}\circ\gamma_{H_{\mp}}=\gamma_{H_{\mp}'}$ and $\widetilde{H_{\mp}}=\widetilde{H_{\mp}'}$. }
\end{align*}
The existence condition of $\xi_{\pm}$ follows from Proposition~\ref{NikProp192}. 
Thus we have a one-to-one correspondence between $\{\gamma_{H_{\mp}}\}/{\sim}$ and $\{\widetilde{H_{\mp}}\}$. 

\begin{lemma}\label{H_-/wt{H_-}}
We have an equality 
\[|H_-|/|\widetilde{H_-}|=|H_+|/|\widetilde{H_+}|. \]
\end{lemma}

\begin{proof}
It is easy to check that 
\[|\Gamma_{L_+S_-}|/|\Gamma_{K_+S_-}|=|\Gamma_{S_+L_-}|/|\Gamma_{S_+K_-}|. \]
Hence the primitivity shows the lemma. 
\end{proof}

\begin{lemma}\label{all 2-elm}
Let $\lambda\in K_+^*$, $\mu\in S_+^*$, $\nu\in S_-^*$. 
If $\lambda+\mu+\nu\in L$, then $\lambda\in \frac{1}{2}K_+$, $\mu\in \frac{1}{2}S_+$, $\nu\in \frac{1}{2}S_-$. 
\end{lemma}

\begin{proof}
Let $T$ be a primitive sublattice of $L$ spanned by $K_+\oplus S_-$, that is, $T=(K_+\oplus S_-)^{\wedge}$. 
Since $T$ is also the fixed part of the action of the involution $(g\circ \epsilon)^*$ on $L$, it follows that $L/(T\oplus T^{\perp})$ is a $2$-elementary group. 
Hence we have $2(\lambda+\nu)+2\mu\in T\oplus T^{\perp}$, in particular $2\mu\in T^{\perp}\subset L$. 
Since $S_+$ is a primitive sublattice of $L$, we see that $2\mu\in S_+^*\cap L \subset (S_+)_L^{\wedge}=S_+$. 
We thus get $\mu\in \frac{1}{2}S_+$. 
The rest of the proof is left to the reader. 
\end{proof}

From this lemma, we see that 
\begin{equation}\label{half}
\gamma_{H_-}(H_-)\subset (\tfrac{1}{2}K_+/K_+ \oplus \tfrac{1}{2}S_+/S_+)/\Gamma_{K_+S_+}. 
\end{equation}

\begin{lemma}\label{wt{H_-}=H_-}
We have $\widetilde{H_{\pm}}=H_{\pm}$ unless $S_{\pm}=D_4(2)$. 
\end{lemma}

\begin{proof}
In cases $S_-(\frac{1}{2})=A_1^4,\ A_1^3,\ A_1^2,\ A_1,\ \{0\}$, 
it follows from (\ref{wt{H_-}}) and Lemma~\ref{Gamma_{pm}=H_{pm}} that $\widetilde{H_-}=H_-$. 
In cases $S_-(\frac{1}{2})=E_8,\ E_7,\ D_6,\ D_4\oplus A_1$, 
we have $\gamma_{H_-}(\Gamma_-)\equiv \frac{1}{2}S_+/S_+ \pmod{\Gamma_{K_+S_+}}$ by Lemma~\ref{Gamma_{pm}=H_{pm}}. 
From (\ref{half}), we see that $\widetilde{H_-}=H_-$. 
It follows from Lemma~\ref{H_-/wt{H_-}} that $\widetilde{H_+}=H_+$. 
\end{proof}

\begin{theorem}\label{Our lists}
The lists~$(\ref{list})$ are classified as Table~$\ref{table1}$ and Table~$\ref{table2}$ in Theorem~$\ref{MainThm}$. 
\end{theorem}

\begin{proof}
By Lemmas \ref{H_{pm}} and \ref{wt{H_-}=H_-}, we calculate $(H_-, K_+, K_-)$ for each $(S_+, S_-)$ except the case $S_{\pm}=D_4(2)$. 
In case $S_{\pm}=D_4(2)$, we have to calculate $(H_-, \widetilde{H_-}, K_+, K_-)$. 
We first calculate $H_-$. 

In case $S_- = E_8(2)$, we see that $q_{S_-}|_{\frac{1}{2}S_-/S_-}=u^4$. 
By Lemma~\ref{H_{pm}}~(3), $\rank H_-=8 \text{ or } 7$. 
For $\rank H_-=8$, we have $H_-=\frac{1}{2}S_-/S_-$. 
For $\rank H_-=7$, we have $q_{S_-}|_{H_-}=u^3\oplus w \text{ or } u^3\oplus z$ by Lemma~\ref{H_{pm}}~(2). 

In case $S_- = E_7(2)$, we see that $q_{S_-}|_{\frac{1}{2}S_-/S_-}=u^3\oplus w$ and $q_{S_{\pm}}|_{\Gamma_{\pm}}=w$. 
By Lemma~\ref{H_{pm}}~(3), $\rank H_-=7 \text{ or } 6$. 
For $\rank H_-=7$, we have $H_-=\frac{1}{2}S_-/S_-$. 
For $\rank H_-=6$, we have $q_{S_-}|_{H_-}=u^2\oplus w^2$ by Lemma~\ref{H_{pm}}~(2) and (4) 
(note that we have $w\oplus z =w^2$). 
The same proof works for the cases $S_-(\frac{1}{2})=D_6,\ D_4 \oplus A_1$. So we omit it. 

In cases $S_-(\frac{1}{2}) = A_1^4,\ A_1^3,\ A_1^2,\ A_1,\ \{0\}$, 
we see that $q_{S_-}|_{H_-}=q_{S_-}|_{\frac{1}{2}S_-/S_-}$ by Lemma~\ref{Gamma_{pm}=H_{pm}}. 

We next deal with the case $S_{\pm} = D_4(2)$. 
We see that $q_{S_-}|_{\frac{1}{2}S_-/S_-}=v\oplus z^2$ and $q_{S_{\pm}}|_{\Gamma_{\pm}}=z^2$. 
By Lemma~\ref{H_{pm}}~(3), $\rank H_-=4 \text{ or } 3$. 
For $\rank H_-=3$, we have $q_{S_-}|_{H_-}=w\oplus z^2$ by Lemma~\ref{H_{pm}}~(2) and (4). 
From (\ref{wt{H_-}}), we have $q_{S_-}|_{\widetilde{H_-}}=w\oplus z^2 \text{ or } z^2$. 
For $\rank H_-=4$, we have $H_-=\frac{1}{2}S_-/S_-$. 
From (\ref{wt{H_-}}), a candidate for $q_{S_-}|_{\widetilde{H_-}}$ is one of $v\oplus z^2$, $w\oplus z^2$ and $z^2$. 
Here we claim that $q_{S_-}|_{\widetilde{H_-}}=z^2$ is impossible. 

Suppose that $q_{S_-}|_{\widetilde{H_-}}=z^2$. 
This yields $\widetilde{H_-}=\Gamma_-$. 
Let $H_-=\widetilde{H_-}\oplus G_{H_-}$, where $G_{H_-}$ is a subgroup of $H_-$ whose quadratic form is $v$. 
Moreover let
\begin{align*}
\tfrac{1}{2}K_+/K_+&=G_{K_+}\oplus p_{K_+}(\Gamma_{K_+S_+}), \\
\tfrac{1}{2}S_+/S_+&=G_{S_+}\oplus p_{S_+}(\Gamma_{K_+S_+}), 
\end{align*}
where $G_{K_+}$ (resp.~$G_{S_+}$) is a subgroup of $\frac{1}{2}K_+/K_+$ (resp.~$\frac{1}{2}S_+/S_+$) whose quadratic form is $u\oplus v$ (resp.~$v$). 
Since $\widetilde{H_-}=\Gamma_-$, we have 
\[\gamma_{H_-}(\widetilde{H_-})\equiv p_{S_+}(\Gamma_{K_+S_+})\equiv p_{K_+}(\Gamma_{K_+S_+}) \pmod{\Gamma_{K_+S_+}}. \]
It follows from (\ref{half}) that 
\[\gamma_{H_-}(G_{H_-})\subset G_{K_+}\oplus G_{S_+}. \]
Since $G_{H_-}$ gives difference between $\Gamma_{K_+S_-}$ and $\Gamma_{L_+S_-}$, 
a non-zero element of $\gamma_{H_-}(G_{H_-})$ is a sum of non-zero elements of $G_{K_+}$ and $G_{S_+}$. 
This contradicts the fact that the quadratic form of $G_{H_-}$ is $v$. 
Now we have Table~\ref{table1}. 

We proceed to calculate $K_{\pm}$. 
By Lemma~\ref{unique-emb}, $K_+(\frac{1}{2})$ is uniquely determined with the signature $(1, 9-\rank S_+)$ and the discriminant form $-q_{S_+(\frac{1}{2})}$. 
By calculating (\ref{k_{pm}}) we have $k_-$. 
From \cite[Theorem~1.14.2 and Corollary~1.9.4]{Nik2}, $K_-$ is uniquely determined with the signature $(2, 10-\rank S_-)$ and the discriminant form $k_-$. 
Therefore we have $k_-$ and $K_-$ in Table~\ref{table2}. 
\end{proof}

%%%%%%%%%%%%%%%%%%%%%%%%%%%%%%%%%%%%%%%%%%%%%%%%%%%

%%%%%%%%%%%%%%%%%%%%%%%%%%%%%%%%%%%%%%%%%%%%%%%%%%
\section{Examples}\label{examples}

In this section we construct examples of involutions on Enriques surfaces. In particular 
we show that all cases in Theorem~\ref{MainThm} actually occur.
We denote by $\iota$ an involution on an Enriques surface $Y$. 
The $K3$-cover is denoted by $X$ with the covering transformation $\varepsilon$. The symplectic lift of $\iota$ to $X$ is denoted by $g$ 
and the other non-symplectic one is $\theta =g\circ \varepsilon =\varepsilon\circ g$.

We first note that the fixed locus of $\theta$, 
\[X^{\theta}=\{x\in X\mid \theta (x)=x\},\]
can be computed from Theorem~\ref{Our lists} via the following theorem.
\begin{theorem}[{\cite[Theorem~4.2.2]{Nik3}}]\label{NikThm422}
Let $\theta$ be a non-symplectic involution of $X$ and let $T=H^2(X, \mathbb{Z})^{\langle\theta^*\rangle}$. 
Since $T$ is $2$-elementary, the lattice $T$ is determined by invariants $(r, l, \delta)$ by Proposition $\ref{NikThm362}$. 
Then, the fixed locus $X^{\theta}$ has the following form. 
\[X^{\theta}=\begin{cases}
C^{(g)}+\sum_{i=1}^k E_i & \text{where } g=\frac{22-r-l}{2} \text{ and } k=\frac{r-l}{2} \\
C_1^{(1)}+C_2^{(1)} & \text{if } r=10,\ l=8,\ \delta=0 \\
\emptyset & \text{if } r=10,\ l=10,\ \delta=0
\end{cases}. \]
Here we denote by $C^{(g)}$ a non-singular curve of genus $g$ and by $E_i$ a non-singular rational curve. 
\end{theorem}
\begin{proposition}
The invariant $(r,l,\delta )$ for each case is as in Table~$\ref{table2}$.
\end{proposition}
\begin{proof}
We see that $T=H^2(X, \mathbb{Z})^{\langle\theta^*\rangle}$ is exactly the sublattice 
$(K_+\oplus S_-)^{\wedge}=((K_+\oplus S_-)\otimes \mathbb{Q})\cap L$ 
of $L=H^2(X, \mathbb{Z})$. Therefore we get $r=\rank K_+ + \rank S_-$. 

Since $T$ is $2$-elementary, we have $\det T=2^l$. 
By $p_{S_-}(\Gamma_{K_+S_-})=\widetilde{H_-}$, it follows that 
\[|\widetilde{H_-}|=|\Gamma_{K_+S_-}|=\sqrt{\frac{\det(K_+\oplus S_-)}{\det(K_+\oplus S_-)^{\wedge}}}=\sqrt{\frac{\det(K_+\oplus S_-)}{2^l}}. \]
From this equation we get $l$. 

Next we compute the invariant $\delta$. In cases No.~[4], [5], [8], [9], [15]--[17], the invariants $(r,l)$ already determine
$\delta$ uniquely by the existence condition for the $2$-elementary hyperbolic lattices, see \cite{Nik3}. 
In cases No.~[1]--[3], [18], we have that the parity of $K_+\oplus S_-$ is zero, hence the overlattice $T$ has parity zero, too. 
In No.~[6], we see from Table \ref{table1} that the length of $\widetilde{H_-}$ is $6$, which equals the rank of $S_-$. 
By straightforward computations, we see that the discriminant group of $T$ has elements of non-integer square, that is, 
we have $\delta=1$ in this case.
In No.~[7], we see that $T^{\perp}$ has rank $8$, signature $(2,6)$ and length $8$.
Therefore  $T^{\perp}(\frac{1}{2})$ is an integral unimodular lattice, which must be odd by the signature reason.
We get $T^{\perp}\simeq A_1(-1)^2\oplus A_1^6$ and so $\delta=1$. 
%In No.~[7], we see from Table \ref{table1} that $|\widetilde{H_-}|=2^6$. 
%By Lemma~\ref{H_-/wt{H_-}}, we have $|\widetilde{H_+}|=2^2$. 
%Hence the length of $\widetilde{H_+}$ is $2$, which equals the rank of $S_+$. 
%Hence the overlattice structure of $T^{\perp}\supset S_+\oplus K_-$ is unique. 
%A direct computation shows $\delta =1$. 

The remaining five cases where $\rank S_+=\rank S_-=4$ are treated by the next two lemmas.
\begin{lemma}\label{temp}
Assume that $S_{\pm}=A_1(2)^4$ and $(r, l)=(10, 10)$. 
Then $T=U(2)\oplus A_1^8$ and $\delta =1$. 
\end{lemma}
\begin{proof}
Let $K_+=U(2)\oplus A_1(2)^4=U(2)\oplus \langle e_1\rangle \oplus \cdots \oplus \langle e_4\rangle$ 
where $e_i$ are generators of $A_1(2)$ respectively. 
Similarly let 
\begin{align*}
S_+=A_1(2)^4=\langle e_1'\rangle \oplus \cdots \oplus \langle e_4'\rangle, \\
S_-=A_1(2)^4=\langle e_1''\rangle \oplus \cdots \oplus \langle e_4''\rangle. 
\end{align*}
By $p_{S_-}(\Gamma_{S_+S_-})=\Gamma_-=\langle e_1''/2\rangle \oplus \cdots \oplus \langle e_4''/2\rangle$, 
elements of norm $1\pmod{2}$ in $\Gamma_-$ is of the form either $e_i''/2$ or $(e_j''+e_k''+e_l'')/2$. 
Hence $\gamma\colon \Gamma_+\to \Gamma_-$ maps $e_i'/2$ to either $e_j''/2$ or $(e_j''+e_k''+e_l'')/2$. 
In the former case, it contradicts the fact that $S=E_8(2)$ does not contain $(-2)$-vector. 
Similarly the patching $p_{S_+}(\Gamma_{K_+S_+})\to p_{K_+}(\Gamma_{K_+S_+})$ maps $e_i'/2$ to $(e_j+e_k+e_l)/2$. 
Hence $\Gamma_{K_+S_-}$ contains an element of the form of 
\[\frac{e_i+e_j+e_k+e_l''+e_m''+e_n''}{2}. \]
This element has norm $(-6)$. 
Assumption $(r, l)=(10, 10)$ yields that $T(\frac{1}{2})=U\oplus E_8$ or $U\oplus\langle -1 \rangle^8$. 
Since $U\oplus E_8$ does not contain $(-3)$-vector, we conclude $T=U(2)\oplus A_1^8$. 
\end{proof}
\begin{lemma}\label{ParityOfT}
Assume that $S_{\pm}=D_4(2)$. 
Then the parity $\delta$ of $T=(K_+\oplus S_-)^{\wedge}$ is equal to $0$. 
\end{lemma}
\begin{proof}
By Corollary~\ref{K_+=U(2)+S_-}, we see that $K_+=U(2)\oplus D_4(2)$. 
Let $q_{K_+}=u\oplus v\oplus v(4)=u\oplus \langle e_1, f_1\rangle \oplus \langle g_1, h_1\rangle$ 
where $\langle e_1, f_1\rangle$ and $\langle g_1, h_1\rangle$ are generators of $v$ and $v(4)$ respectively. 
Similarly, let 
\begin{align*}
q_{S_+}&=v\oplus v(4)=\langle e_2, f_2\rangle \oplus \langle g_2, h_2\rangle, \\
q_{S_-}&=v\oplus v(4)=\langle e_3, f_3\rangle \oplus \langle g_3, h_3\rangle. 
\end{align*}
Recall that $L_+=U(2)\oplus E_8(2)$ and $S=E_8(2)$. 
We see that $\Gamma_{K_+S_+}=\langle 2g_1+2g_2, 2h_1+2h_2\rangle$ and $\Gamma_{S_+S_-}=\langle 2g_2+2g_3, 2h_2+2h_3\rangle$. 
Hence $\Gamma_{K_+S_-}$ contains $\langle 2g_1+2g_3, 2h_1+2h_3\rangle$. 
This shows that $T$ is an overlattice of $U(2)\oplus E_8(2)$. 
Therefore the parity of $T$ is equal to $0$. 
\end{proof}
This completes the proofs for all cases.
\end{proof}

\subsection{Horikawa constructions} 

The general construction is as follows.
\begin{proposition}[{\cite[V.~23]{BHPV}}]
Let $\psi$ be an involution on $\mathbb{P}^1\times \mathbb{P}^1$ given by $\psi\colon (u, v) \mapsto (-u, -v)$ 
where $u$ and $v$ are inhomogeneous coordinates of $\mathbb{P}^1$ respectively. 
Let $B$ be a curve on $\mathbb{P}^1\times \mathbb{P}^1$ whose bidegree is $(4, 4)$ with at worst simple singularities and 
preserved under $\psi$. 
Assume that $B$ does not pass through any of fixed points of $\psi$. 
Then the minimal resolution $X$ of the double cover of $\mathbb{P}^1\times \mathbb{P}^1$ whose branch locus is $B$
is a $K3$ surface. 
Moreover, $\psi$ lifts to two involutions of $X$. 
One of them is a fixed point free involution $\varepsilon$. 
In particular, $Y=X/\varepsilon$ is an Enriques surface. 
\end{proposition}

In this construction, the other lift of $\psi$ gives a symplectic involution $g$ on $X$ and induces an involution $\iota$ on $Y$
(namely the construction always associates an involution on $Y$). 
The covering involution $\theta$ of $X/\mathbb{P}^1\times \mathbb{P}^1$ is the same as $\varepsilon\circ g$, which is a non-symplectic involution of $X$. In what follows, we exhibit many choices of branch $B$ so that 
the resulting $\iota$ covers all involutions in Theorem \ref{MainThm} except for No.~[13].
We remark that, the condition for $B$ to have the expected number of components, types of singularities and 
not to pass through the fixed points of $\psi$ is Zariski open, so that we will always assume that the 
coefficients (parameters) of the exhibited equation of $B$ are general enough to satisfy these conditions.

%Recall that 
%\[L=H^2(X, \mathbb{Z}),\quad L^{\langle\epsilon^*\rangle}=L_+,\quad L_{\langle\epsilon^*\rangle}=L_-,\quad L_{\langle g^*\rangle}=S. \]
%Let $T=L^{\langle\theta^*\rangle}$. 
%We see that $T=(K_+\oplus S_-)^{\wedge}$. 
%Note that $T$ is a $2$-elementary lattice, since $\theta^*$ is an involution. 

%\begin{lemma}\label{SizeOfH_-}
%We have an equality 
%\[|\widetilde{H_-}|=\sqrt{\frac{\det(K_+\oplus S_-)}{2^l}}. \]
%\end{lemma}
%\begin{proof}
%Since $T=H^2(X, \mathbb{Z})^{\langle\theta^*\rangle}$ is a $2$-elementary lattice, we have $\det T=2^l$. 
%By $p_{S_-}(\Gamma_{K_+S_-})=\widetilde{H_-}$, it follows that 
%\[|\widetilde{H_-}|=|\Gamma_{K_+S_-}|=\sqrt{\frac{\det(K_+\oplus S_-)}{\det(K_+\oplus S_-)^{\wedge}}}=\sqrt{\frac{\det(K_+\oplus S_-)}{2^l}}. \]
%\end{proof}

\medskip
\noindent
\textbf{Example No.~[1].}\quad 
This example was constructed by Horikawa~\cite{Hor}, and studied by Dolgachev~\cite{Dol} and Barth-Peters~\cite{BP}. 
Here we give another construction given by Mukai-Namikawa~\cite{MN}. 

Consider the following curves on $\mathbb{P}^1\times\mathbb{P}^1$ (Figure~\ref{1-1}); 
\begin{align*}
&X_{\pm}\colon u=\pm 1,\ Y_{\pm}\colon v=\pm 1, \\
&E\colon u^2v^2-1+a_1(u^2-1)+a_2(v^2-1)=0 \pod{a_i\in\mathbb{C}}. 
\end{align*}

%We assume that coefficients $(a_1, a_2)$ satisfies the conditions such that $E$ is smooth and irreducible, 
%and $E$ does not pass through any of fixed points of $\psi$. 
%The set of $(a_1, a_2)$ satisfying these conditions forms an open set of $\mathbb{C}^2$. 
%Similarly, we assume that $a_i$ satisfies these conditions in other examples. 

\begin{figure}[htbp]
 \begin{minipage}{0.45\hsize}
  \begin{center}
   \includegraphics[width=40mm]{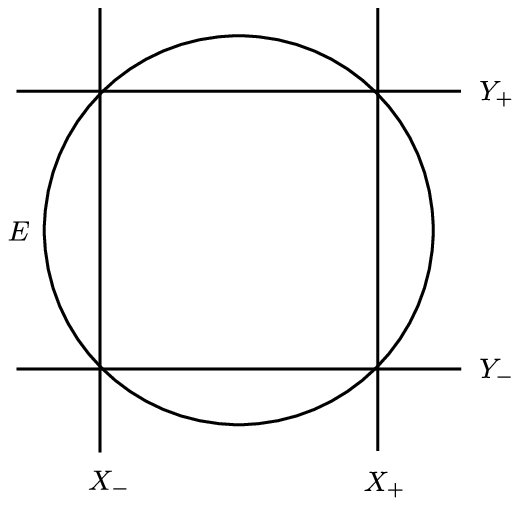}
  \end{center}
  \caption{}
  \label{1-1}
 \end{minipage}
 \begin{minipage}{0.45\hsize}
  \begin{center}
   \includegraphics[width=40mm]{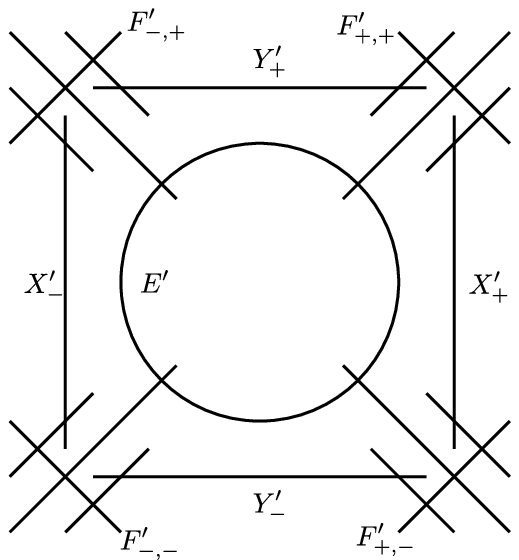}
  \end{center}
  \caption{}
  \label{1-2}
 \end{minipage}
\end{figure}

Blow up $\mathbb{P}^1\times\mathbb{P}^1$ at $4$ intersection points of $X_{\pm}$, $Y_{\pm}$ and $E$. 
Let $F_{\pm, \pm}$ be the exceptional curves over $(\pm 1, \pm 1)$ respectively. 
Blow up again at $12$ intersection points of $F_{\pm, \pm}$ and the strict transforms of $X_{\pm}$, $Y_{\pm}$ and $E$. 
Let $R$ be the blown up surface. 
We denote by $X_{\pm}'$, $Y_{\pm}'$, $F_{\pm, \pm}'$ and $E'$ the strict transforms of $X_{\pm}$, $Y_{\pm}$, $F_{\pm, \pm}$ and $E$ respectively. 
The configuration of curves in $R$ is given in Figure~\ref{1-2}. 
Note that $X_{\pm}'$, $Y_{\pm}'$ and $F_{\pm, \pm}'$ are all $(-4)$-curves, and other rational curves are all $(-1)$-curves. 
Let $B'=\sum (X_{\pm}'+Y_{\pm}'+F_{\pm, \pm}')+E'$. 
The $K3$ surface $X$ is the double cover of $R$ whose branch locus is $B'$. 
Since $X^{\theta}=B'$ consists of one elliptic curve and $8$ rational curves, we see $(r, l)=(18, 2)$, by Theorem~\ref{NikThm422}. 
This is enough to conclude that this example belongs to No.~[1] by Table~\ref{table2}.
%By $r=18$, we see $\rank S_-=8$. 
%This shows $S_-=E_8(2)$ by Theorem~\ref{Our lists}. 
%By Lemma~\ref{SizeOfH_-}, we have 
%\[|\widetilde{H_-}|=\sqrt{\frac{2^{18}}{2^2}}=2^8. \]
%Therefore this is the example of No.~[1]. 

\medskip
\noindent
\textbf{Example No.~[2].}\quad  
This example was found by Kondo, and overlooked in \cite{MN} (cf.~\cite{Muk1}). 

Consider the following curves on $\mathbb{P}^1\times\mathbb{P}^1$ (Figure~\ref{2-1}); 
\begin{align*}
&X_{\pm}\colon u=\pm 1,\ Y_{\pm}\colon v=\pm 1, \\
&C_{\pm}\colon uv-1+a_1(\pm u-1)+a_2(\pm v-1)=0 \pod{a_i\in\mathbb{C}}. 
\end{align*}

\begin{figure}[htbp]
 \begin{minipage}{0.45\hsize}
  \begin{center}
   \includegraphics[width=40mm]{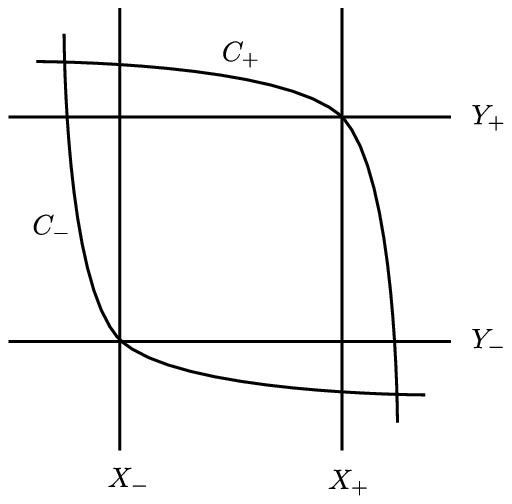}
  \end{center}
  \caption{}
  \label{2-1}
 \end{minipage}
 \begin{minipage}{0.45\hsize}
  \begin{center}
   \includegraphics[width=40mm]{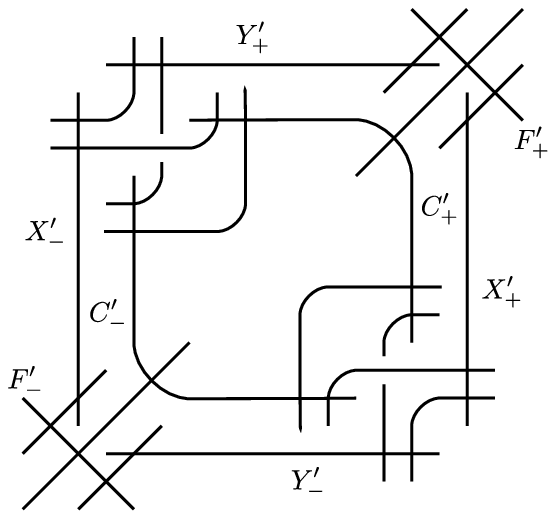}
  \end{center}
  \caption{}
  \label{2-2}
 \end{minipage}
\end{figure}

\begin{figure}[htbp]
  \begin{center}
   \includegraphics[width=40mm]{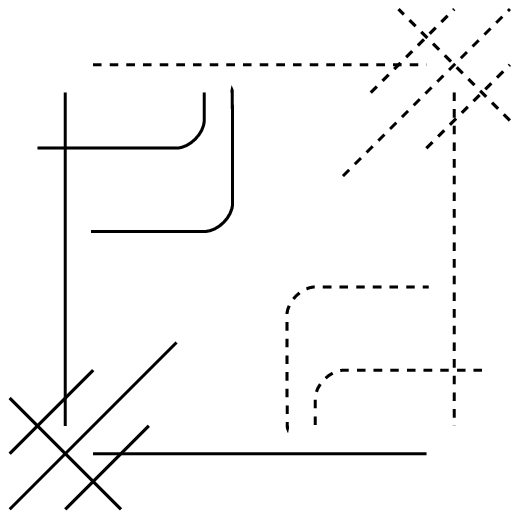}
  \end{center}
  \caption{}
  \label{2-3}
\end{figure}

Blow up $\mathbb{P}^1\times\mathbb{P}^1$ at $10$ intersection points of $X_{\pm}$, $Y_{\pm}$ and $C_{\pm}$. 
Let $F_+$ and $F_-$ be the exceptional curves over $(1, 1)$ and $(-1, -1)$ respectively. 
Blow up again at $6$ intersection points of $F_{\pm}$ and the strict transforms of $X_{\pm}$, $Y_{\pm}$ and $C_{\pm}$. 
Let $R$ be the blown up surface. 
We denote by $X_{\pm}'$, $Y_{\pm}'$, $C_{\pm}'$ and $F_{\pm}'$ the strict transforms of $X_{\pm}$, $Y_{\pm}$, $C_{\pm}$ and $F_{\pm}$ respectively. 
The configuration of curves in $R$ is given in Figure~\ref{2-2}. 
Note that $X_{\pm}'$, $Y_{\pm}'$, $C_{\pm}'$ and $F_{\pm}'$ are all $(-4)$-curves, and the others are all $(-1)$-curves. 
Let $B'=\sum (X_{\pm}'+Y_{\pm}'+C_{\pm}'+F_{\pm}')$. 
The $K3$ surface $X$ is the double cover of $R$ whose branch locus is $B'$. 
Since $X^{\theta}=B'$ consists of $8$ rational curves, we see $(r, l)=(18, 4)$, by Theorem~\ref{NikThm422}. 
%By $r=18$, we see $\rank S_-=8$. 
%This shows $S_-=E_8(2)$ by Theorem~\ref{Our lists}. 
%By Lemma~\ref{SizeOfH_-}, we have 
%\[|\widetilde{H_-}|=\sqrt{\frac{2^{18}}{2^4}}=2^7. \]
Note that the configuration of curves in $X$ is given in the same as Figure~\ref{2-2}. 
We notice that there exists $E_7\oplus A_1$ diagram in Figure~\ref{2-2} (continuous lines in Figure~\ref{2-3}). 
Let $e_i \pod{i=1,\ldots, 8}$ denote the cohomology class of these curves respectively. 
The image of this diagram by $\epsilon$ is given by dashed lines in Figure~\ref{2-3}. 
Let $M$ be the lattice generated by $e_i-\epsilon^*(e_i) \pod{i=1,\ldots, 8}$. 
We see that $M\cong E_7(2)\oplus A_1(2)$ and $M\subset S_-$. 
For $(e_i-\epsilon^*(e_i))/2\in\frac{1}{2}M$, there exists $(e_i+\epsilon^*(e_i))/2\in L_+^*$ such that 
\[\frac{e_i-\epsilon^*(e_i)}{2}+\frac{e_i+\epsilon^*(e_i)}{2}=e_i\in L. \]
It follows that 
\[\tfrac{1}{2}M/S_-\subset H_-. \]
By calculation, we have $q_{E_8(2)}|_{\frac{1}{2}(E_7(2)\oplus A_1(2))/E_8(2)}=u^3\oplus w$. 
Therefore this is the example of No.~[2].

\medskip
\noindent
\textbf{Example No.~[3].}\quad 
This example was constructed by Lieberman. 

Consider the following curves on $\mathbb{P}^1\times\mathbb{P}^1$ (Figure~\ref{3-1}); 
\begin{align*}
&X_{1\pm}\colon u=\pm 1,\ Y_{1\pm}\colon v=\pm 1, \\
&X_{2\pm}\colon u=\pm a_1,\ Y_{2\pm}\colon v=\pm a_2 \pod{a_i\in\mathbb{C}}. 
\end{align*}

\begin{figure}[htbp]
 \begin{minipage}{0.45\hsize}
  \begin{center}
   \includegraphics[width=40mm]{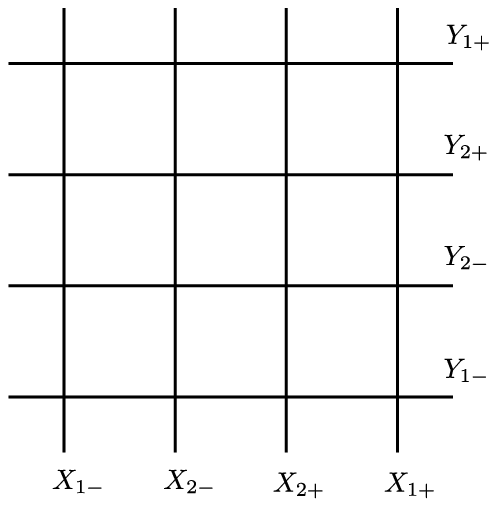}
  \end{center}
  \caption{}
  \label{3-1}
 \end{minipage}
 \begin{minipage}{0.45\hsize}
  \begin{center}
   \includegraphics[width=40mm]{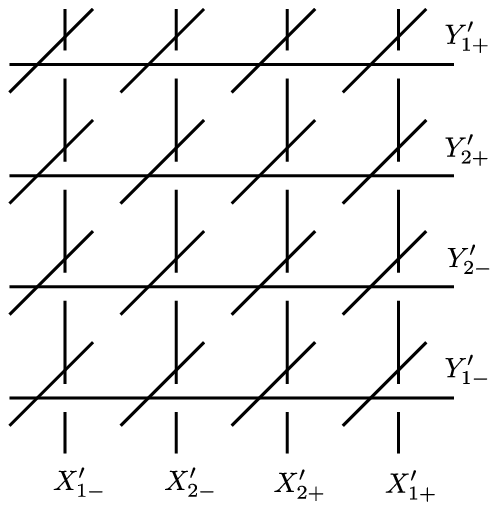}
  \end{center}
  \caption{}
  \label{3-2}
 \end{minipage}
\end{figure}

\begin{figure}[htbp]
  \begin{center}
   \includegraphics[width=40mm]{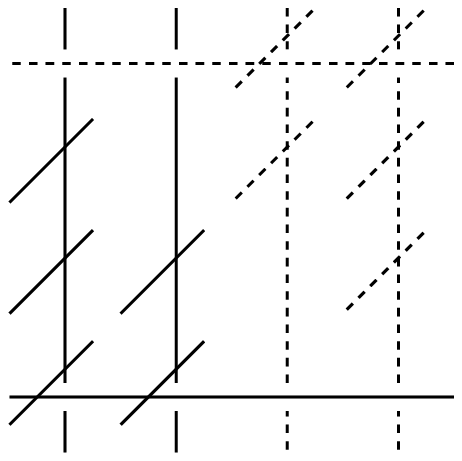}
  \end{center}
  \caption{}
  \label{3-3}
\end{figure}

Blow up $\mathbb{P}^1\times\mathbb{P}^1$ at $16$ intersection points of $X_{1\pm}$, $X_{2\pm}$, $Y_{1\pm}$ and $Y_{2\pm}$. 
Let $R$ be the blown up surface. 
We denote by $X_{1\pm}'$, $X_{2\pm}'$, $Y_{1\pm}'$ and $Y_{2\pm}'$ the strict transforms of $X_{1\pm}$, $X_{2\pm}$, $Y_{1\pm}$ and $Y_{2\pm}$ respectively. 
The configuration of curves in $R$ is given in Figure~\ref{3-2}. 
Note that $X_{1\pm}'$, $X_{2\pm}'$, $Y_{1\pm}'$ and $Y_{2\pm}'$ are all $(-4)$-curves, and the others are all $(-1)$-curves. 
Let $B'=\sum (X_{1\pm}'+X_{2\pm}'+Y_{1\pm}'+Y_{2\pm}')$. 
The $K3$ surface $X$ is the double cover of $R$ whose branch locus is $B'$. 
Since $X^{\theta}=B'$ consists of $8$ rational curves, we see $(r, l)=(18, 4)$, by Theorem~\ref{NikThm422}. 
%By $r=18$, we see $\rank S_-=8$. 
%This shows $S_-=E_8(2)$ by Theorem~\ref{Our lists}. 
%By Lemma~\ref{SizeOfH_-}, we have 
%\[|\widetilde{H_-}|=\sqrt{\frac{2^{18}}{2^4}}=2^7. \]
Note that the configuration of curves in $X$ is given in the same as Figure~\ref{3-2}. 
We notice that there exists $D_8$ diagram in Figure~\ref{3-2} (continuous lines in Figure~\ref{3-3}). 
Let $e_i \pod{i=1,\ldots, 8}$ denote the cohomology class of these curves respectively. 
The image of this diagram by $\epsilon$ is given by dashed lines in Figure~\ref{3-3}. 
Let $M$ be the lattice generated by $e_i-\epsilon^*(e_i) \pod{i=1,\ldots, 8}$. 
We see that $M\cong D_8(2)$ and $M\subset S_-$. 
Similarly to the Example~No.~[2], we have $\tfrac{1}{2}M/S_-\subset H_-$. 
By calculation, we have $q_{E_8(2)}|_{\frac{1}{2}(D_8(2))/E_8(2)}=u^3\oplus z$. 
Therefore this is the example of No.~[3].

\medskip
\noindent
\textbf{Example No.~[4].}\quad 
Consider the following curves on $\mathbb{P}^1\times\mathbb{P}^1$ (Figure~\ref{4-1}); 
\begin{align*}
&X_{\pm}\colon u=\pm 1,\ Y_{\pm}\colon v=\pm 1, \\
&E\colon u^2v^2-1+a_1(u^2-1)+a_2(v^2-1)+a_3(uv-1)=0 \pod{a_i\in\mathbb{C}}. 
\end{align*}

\begin{figure}[htbp]
 \begin{minipage}{0.45\hsize}
  \begin{center}
   \includegraphics[width=40mm]{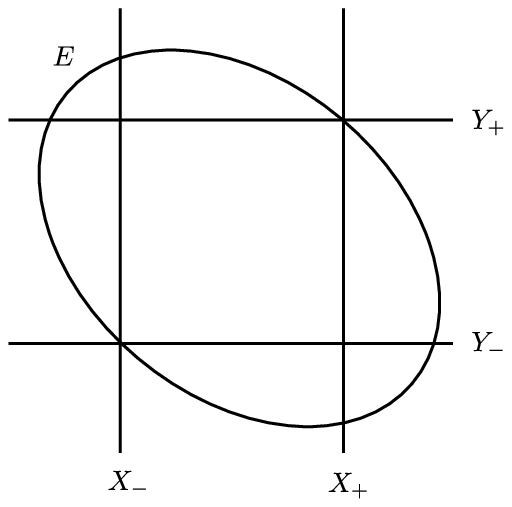}
  \end{center}
  \caption{}
  \label{4-1}
 \end{minipage}
 \begin{minipage}{0.45\hsize}
  \begin{center}
   \includegraphics[width=40mm]{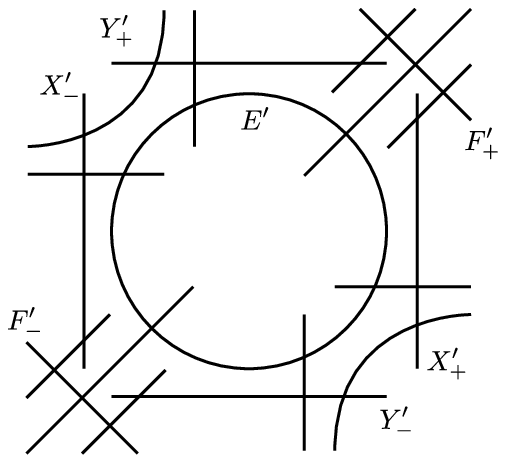}
  \end{center}
  \caption{}
  \label{4-2}
 \end{minipage}
\end{figure}

Blow up $\mathbb{P}^1\times\mathbb{P}^1$ at $8$ intersection points of $X_{\pm}$, $Y_{\pm}$ and $E$. 
Let $F_+$ and $F_-$ be the exceptional curves over $(1, 1)$ and $(-1, -1)$ respectively. 
Blow up again at $6$ intersection points of $F_{\pm}$ and the strict transforms of $X_{\pm}$, $Y_{\pm}$ and $E$. 
Let $R$ be the blown up surface. 
We denote by $X_{\pm}'$, $Y_{\pm}'$, $F_{\pm}'$ and $E'$ the strict transforms of $X_{\pm}$, $Y_{\pm}$, $F_{\pm}$ and $E$ respectively. 
The configuration of curves in $R$ is given in Figure~\ref{4-2}. 
Note that $X_{\pm}'$, $Y_{\pm}'$ and $F_{\pm}'$ are all $(-4)$-curves, and other rational curves are all $(-1)$-curves. 
Let $B'=\sum (X_{\pm}'+Y_{\pm}'+F_{\pm}')+E'$. 
The $K3$ surface $X$ is the double cover of $R$ whose branch locus is $B'$. 
Since $X^{\theta}=B'$ consists of one elliptic curve and $6$ rational curves, we see $(r, l)=(16, 4)$, by Theorem~\ref{NikThm422}. 
%By $r=16$, we see $\rank S_-=7$. 
%This shows $S_-=E_7(2)$ by Theorem~\ref{Our lists}. 
%By Lemma~\ref{SizeOfH_-}, we have 
%\[|\widetilde{H_-}|=\sqrt{\frac{2^{18}}{2^4}}=2^7. \]
Therefore this is the example of No.~[4]. 

\medskip
\noindent
\textbf{Example No.~[5].}\quad 
This example was studied by Mukai~\cite{Muk2} as the example of numerically reflective involution. 

Consider the following curves on $\mathbb{P}^1\times\mathbb{P}^1$ (Figure~\ref{5-1}); 
\begin{align*}
&X_{\pm}\colon u=\pm 1,\ Y_{\pm}\colon v=\pm 1, \\
&C_{\pm}\colon uv\pm a_1u\pm a_2v+a_3=0 \pod{a_i\in\mathbb{C}}. 
\end{align*}

\begin{figure}[htbp]
 \begin{minipage}{0.45\hsize}
  \begin{center}
   \includegraphics[width=40mm]{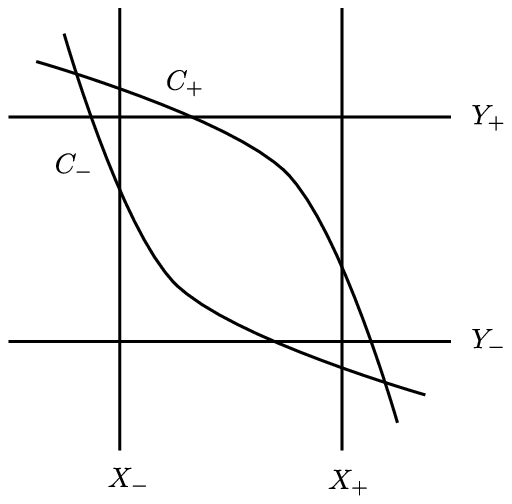}
  \end{center}
  \caption{}
  \label{5-1}
 \end{minipage}
 \begin{minipage}{0.45\hsize}
  \begin{center}
   \includegraphics[width=40mm]{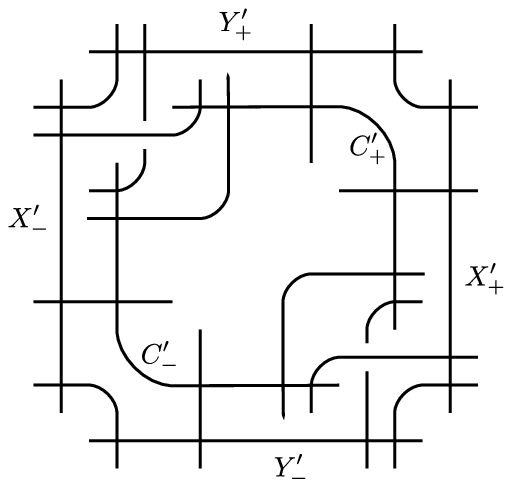}
  \end{center}
  \caption{}
  \label{5-2}
 \end{minipage}
\end{figure}

Blow up $\mathbb{P}^1\times\mathbb{P}^1$ at $14$ intersection points of $X_{\pm}$, $Y_{\pm}$ and $C_{\pm}$. 
Let $R$ be the blown up surface. 
We denote by $X_{\pm}'$, $Y_{\pm}'$ and $C_{\pm}'$ the strict transforms of $X_{\pm}$, $Y_{\pm}$ and $C_{\pm}$ respectively. 
The configuration of curves in $R$ is given in Figure~\ref{5-2}. 
Note that $X_{\pm}'$, $Y_{\pm}'$ and $C_{\pm}'$ are all $(-4)$-curves and the others are all $(-1)$-curves. 
Let $B'=\sum (X_{\pm}'+Y_{\pm}'+C_{\pm}')$. 
The $K3$ surface $X$ is the double cover of $R$ whose branch locus is $B'$. 
Since $X^{\theta}=B'$ consists of $6$ rational curves, we see $(r, l)=(16, 6)$, by Theorem~\ref{NikThm422}. 
%By $r=16$, we see $\rank S_-=7$. 
%This shows $S_-=E_7(2)$ by Theorem~\ref{Our lists}. 
%By Lemma~\ref{SizeOfH_-}, we have 
%\[|\widetilde{H_-}|=\sqrt{\frac{2^{18}}{2^6}}=2^6. \]
Therefore this is the example of No.~[5].

\medskip
\noindent
\textbf{Example No.~[6].}\quad 
Consider the following curves on $\mathbb{P}^1\times\mathbb{P}^1$ (Figure~\ref{6-1}); 
\begin{align*}
&X_{\pm}\colon u=\pm 1,\ Y_{\pm}\colon v=\pm 1, \\
&E\colon u^2v^2+a_1u^2+a_2v^2+a_3uv+a_4=0 \pod{a_i\in\mathbb{C}}. 
\end{align*}

\begin{figure}[htbp]
 \begin{minipage}{0.45\hsize}
  \begin{center}
   \includegraphics[width=40mm]{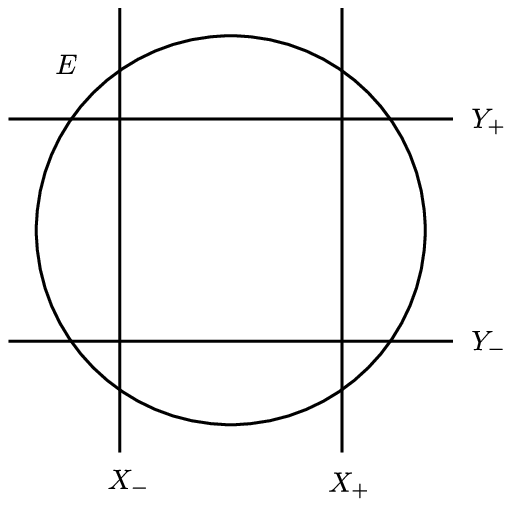}
  \end{center}
  \caption{}
  \label{6-1}
 \end{minipage}
 \begin{minipage}{0.45\hsize}
  \begin{center}
   \includegraphics[width=40mm]{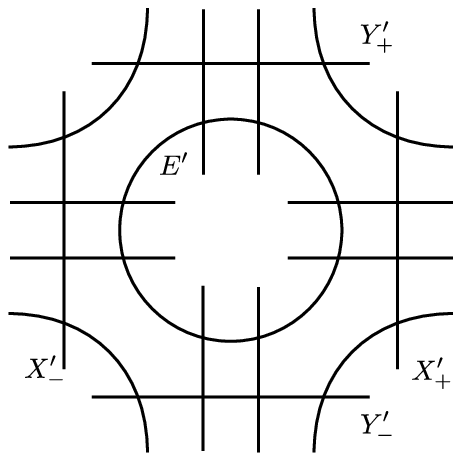}
  \end{center}
  \caption{}
  \label{6-2}
 \end{minipage}
\end{figure}

Blow up $\mathbb{P}^1\times\mathbb{P}^1$ at $12$ intersection points of $X_{\pm}$, $Y_{\pm}$ and $E$. 
Let $R$ be the blown up surface. 
We denote by $X_{\pm}'$, $Y_{\pm}'$ and $E'$ the strict transforms of $X_{\pm}$, $Y_{\pm}$ and $E$ respectively. 
The configuration of curves in $R$ is given in Figure~\ref{6-2}. 
Note that $X_{\pm}'$, $Y_{\pm}'$ are all $(-4)$-curves and other rational curves are all $(-1)$-curves. 
Let $B'=\sum (X_{\pm}'+Y_{\pm}')+E'$. 
The $K3$ surface $X$ is the double cover of $R$ whose branch locus is $B'$. 
Since $X^{\theta}=B'$ consists of one elliptic curve and $4$ rational curves, we see $(r, l)=(14, 6)$, by Theorem~\ref{NikThm422}. 
%By $r=14$, we see $\rank S_-=6$. 
%This shows $S_-=D_6(2)$ by Theorem~\ref{Our lists}. 
%By Lemma~\ref{SizeOfH_-}, we have 
%\[|\widetilde{H_-}|=\sqrt{\frac{2^{18}}{2^6}}=2^6. \]
Therefore this is the example of No.~[6].

\medskip
\noindent
\textbf{Example No.~[7].}\quad 
Consider the following curves on $\mathbb{P}^1\times\mathbb{P}^1$ (Figure~\ref{7-1}); 
\[Y_{\pm}\colon v=\pm 1,\ C_{\pm}\colon u^2v\pm uv\pm a_1u^2+a_2u+a_3v\pm a_4=0 \pod{a_i\in\mathbb{C}}. \]

\begin{figure}[htbp]
 \begin{minipage}{0.45\hsize}
  \begin{center}
   \includegraphics[width=40mm]{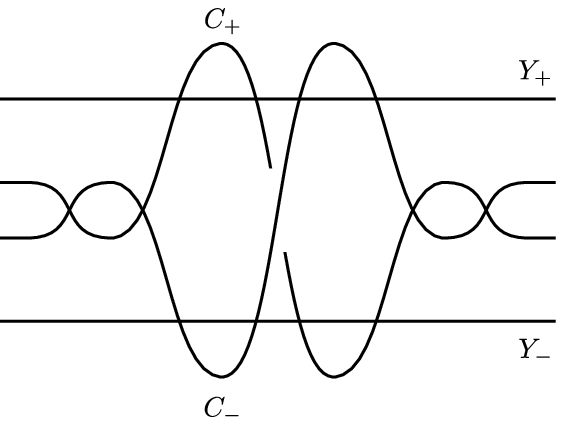}
  \end{center}
  \caption{}
  \label{7-1}
 \end{minipage}
 \begin{minipage}{0.45\hsize}
  \begin{center}
   \includegraphics[width=40mm]{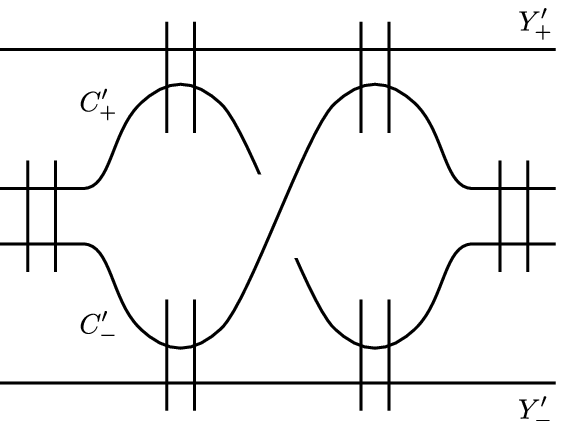}
  \end{center}
  \caption{}
  \label{7-2}
 \end{minipage}
\end{figure}

Blow up $\mathbb{P}^1\times\mathbb{P}^1$ at $12$ intersection points of $Y_{\pm}$ and $C_{\pm}$. 
Let $R$ be the blown up surface. 
We denote by $Y_{\pm}'$ and $C_{\pm}'$ the strict transforms of $Y_{\pm}$ and $C_{\pm}$ respectively. 
The configuration of curves in $R$ is given in Figure~\ref{7-2}. 
Note that $Y_{\pm}'$ and $C_{\pm}'$ are all $(-4)$-curves and the others are all $(-1)$-curves. 
Let $B'=\sum (Y_{\pm}'+C_{\pm}')$. 
The $K3$ surface $X$ is the double cover of $R$ whose branch locus is $B'$. 
Since $X^{\theta}=B'$ consists of $4$ rational curves, we see $(r, l)=(14, 8)$, by Theorem~\ref{NikThm422}. 
%By $r=14$, we see $\rank S_-=6$. 
%This shows $S_-=D_6(2)$ by Theorem~\ref{Our lists}. 
%By Lemma~\ref{SizeOfH_-}, we have 
%\[|\widetilde{H_-}|=\sqrt{\frac{2^{18}}{2^8}}=2^5. \]
Therefore this is the example of No.~[7].

\medskip
\noindent
\textbf{Example No.~[8].}\quad 
Consider the following curves on $\mathbb{P}^1\times\mathbb{P}^1$ (Figure~\ref{8-1}); 
\[Y_{\pm}\colon v=\pm 1,\ E\colon v^2(u^4+a_1u^2+a_2)+2a_3uv(u^2-a_4)+a_5(u^2-a_4)^2=0 \pod{a_i\in\mathbb{C}}. \]
Note that $E$ has $2$ nodes at $(u, v)=(\pm\sqrt{a_4}, 0)$. 

\begin{figure}[htbp]
 \begin{minipage}{0.45\hsize}
  \begin{center}
   \includegraphics[width=40mm]{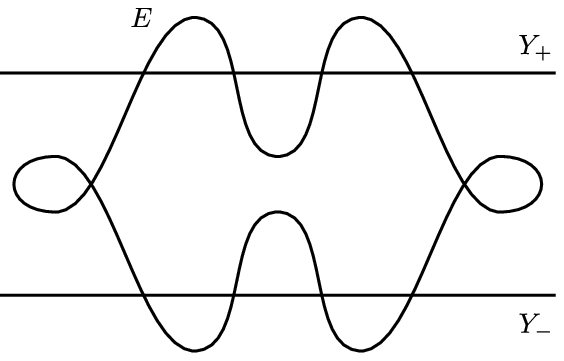}
  \end{center}
  \caption{}
  \label{8-1}
 \end{minipage}
 \begin{minipage}{0.45\hsize}
  \begin{center}
   \includegraphics[width=40mm]{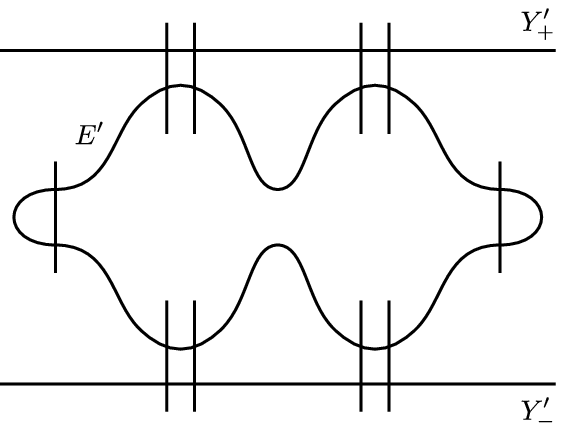}
  \end{center}
  \caption{}
  \label{8-2}
 \end{minipage}
\end{figure}

Blow up $\mathbb{P}^1\times\mathbb{P}^1$ at $8$ intersection points of $Y_{\pm}$ and $E$, and at $2$ nodes of $E$. 
Let $R$ be the blown up surface. 
We denote by $Y_{\pm}'$ and $E'$ the strict transforms of $Y_{\pm}$ and $E$ respectively. 
The configuration of curves in $R$ is given in Figure~\ref{8-2}. 
Note that $Y_{\pm}'$ are $(-4)$-curves and other rational curves are all $(-1)$-curves. 
Let $B'=Y_+'+Y_-'+E'$. 
The $K3$ surface $X$ is the double cover of $R$ whose branch locus is $B'$. 
Since $X^{\theta}=B'$ consists of one elliptic curve and $2$ rational curves, we see $(r, l)=(12, 8)$, by Theorem~\ref{NikThm422}. 
%By $r=12$, we see $\rank S_-=5$. 
%This shows $S_-=D_4(2)\oplus A_1(2)$ by Theorem~\ref{Our lists}. 
%By Lemma~\ref{SizeOfH_-}, we have 
%\[|\widetilde{H_-}|=\sqrt{\frac{2^{18}}{2^8}}=2^5. \]
Therefore this is the example of No.~[8].

\medskip
\noindent
\textbf{Example No.~[9].}\quad 
Consider the following curves on $\mathbb{P}^1\times\mathbb{P}^1$ (Figure~\ref{9-1}); 
\[C_{\pm}\colon v^2(u^2\pm a_1u+a_2)\pm 2a_3v(u\mp a_4)^2+a_5(u\mp a_4)^2=0 \pod{a_i\in\mathbb{C}}. \]
Note that $C_+$ and $C_-$ have a node at $(u, v)=(a_4, 0)$ and $(-a_4, 0)$ respectively.  

\begin{figure}[htbp]
 \begin{minipage}{0.45\hsize}
  \begin{center}
   \includegraphics[width=40mm]{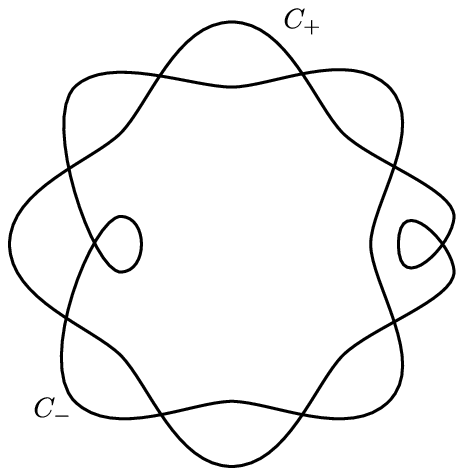}
  \end{center}
  \caption{}
  \label{9-1}
 \end{minipage}
 \begin{minipage}{0.45\hsize}
  \begin{center}
   \includegraphics[width=40mm]{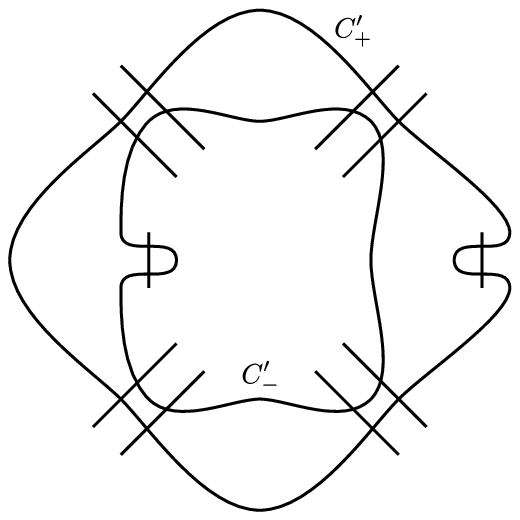}
  \end{center}
  \caption{}
  \label{9-2}
 \end{minipage}
\end{figure}

Blow up $\mathbb{P}^1\times\mathbb{P}^1$ at $8$ intersection points of $C_{\pm}$, and at $2$ nodes of $C_{\pm}$. 
Let $R$ be the blown up surface. 
We denote by $C_{\pm}'$ the strict transforms of $C_{\pm}$ respectively. 
The configuration of curves in $R$ is given in Figure~\ref{9-2}. 
Note that $C_{\pm}'$ are $(-4)$-curves and the others are all $(-1)$-curves. 
Let $B'=C_+'+C_-'$. 
The $K3$ surface $X$ is the double cover of $R$ whose branch locus is $B'$. 
Since $X^{\theta}=B'$ consists of $2$ rational curves, we see $(r, l)=(12, 10)$, by Theorem~\ref{NikThm422}. 
%By $r=12$, we see $\rank S_-=5$. 
%This shows $S_-=D_4(2)\oplus A_1(2)$ by Theorem~\ref{Our lists}. 
%By Lemma~\ref{SizeOfH_-}, we have 
%\[|\widetilde{H_-}|=\sqrt{\frac{2^{18}}{2^{10}}}=2^4. \]
Therefore this is the example of No.~[9].

\medskip
\noindent
\textbf{Example No.~[10].}\quad 
Consider the following curves on $\mathbb{P}^1\times\mathbb{P}^1$ (Figure~\ref{10-1}); 
\[Y_{\pm}\colon v=\pm 1,\ C\colon v^2(u^4+u^2+a_1)+vu(a_2 u^2+a_3)+a_4 u^4+a_5 u^2+a_6=0 \pod{a_i\in\mathbb{C}}. \]

\begin{figure}[htbp]
 \begin{minipage}{0.45\hsize}
  \begin{center}
   \includegraphics[width=40mm]{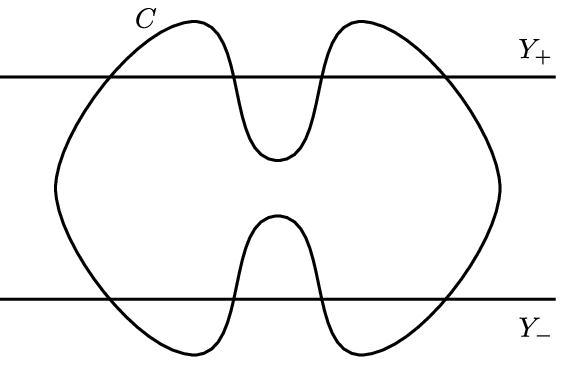}
  \end{center}
  \caption{}
  \label{10-1}
 \end{minipage}
 \begin{minipage}{0.45\hsize}
  \begin{center}
   \includegraphics[width=40mm]{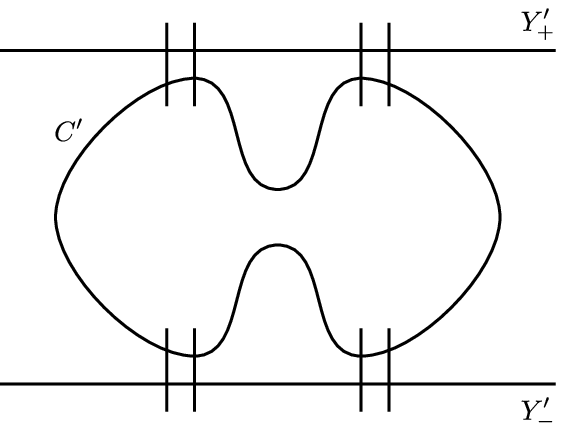}
  \end{center}
  \caption{}
  \label{10-2}
 \end{minipage}
\end{figure}

Blow up $\mathbb{P}^1\times\mathbb{P}^1$ at $8$ intersection points of $Y_{\pm}$ and $C$. 
Let $R$ be the blown up surface. 
We denote by $Y_{\pm}'$ and $C'$ the strict transforms of $Y_{\pm}$ and $C$ respectively. 
The configuration of curves in $R$ is given in Figure~\ref{10-2}. 
Note that $Y_{\pm}'$ are $(-4)$-curves and other rational curves are all $(-1)$-curves. 
Let $B'=Y_+'+Y_-'+C'$. 
The $K3$ surface $X$ is the double cover of $R$ whose branch locus is $B'$. 
Since $X^{\theta}=B'$ consists of a curve of genus $3$ and $2$ rational curves, we see $(r, l)=(10, 6)$, by Theorem~\ref{NikThm422}. 
%By $r=10$, we see $\rank S_-=4$. 
%Note that the configuration of curves in $X$ is given in the same as Figure~\ref{10-2}. 
%We notice that there exist two $D_4$ diagrams which are mapped by $\epsilon$ each other in Figure~\ref{10-2}. 
%Similarly to the Example~No.~[2], we see $D_4(2)\subset S_-$. 
%This shows $S_-=D_4(2)$ by Theorem~\ref{Our lists}. 
%By Lemma~\ref{SizeOfH_-}, we have 
%\[|\widetilde{H_-}|=\sqrt{\frac{2^{14}}{2^6}}=2^4. \]
Therefore this is the example of No.~[10].

\medskip
\noindent
\textbf{Example No.~[11].}\quad 
%Consider the following curves on $\mathbb{P}^1\times\mathbb{P}^1$ (Figure~); 
Consider the following curves on $\mathbb{P}^1\times\mathbb{P}^1$; 
\begin{align*}
&E_1\colon u^2v^2+u^2+a_1v^2+a_2 uv+a_3=0, \\
&E_2\colon u^2v^2+v^2+a_4u^2+a_5 uv+a_6=0 \pod{a_i\in\mathbb{C}}. 
\end{align*}
%Then $E_i$ are smooth elliptic curves and preserved by $\psi$. We can see that $X^{\theta}$ consists of two elliptic curves and 
%therefore $(r,l,\delta )=(10,8,0)$. To see to which No. this example belongs, we argue as follows.
%
%Blowing up the $8$ intersection points of $E_1\cap E_2$, we get a rational elliptic surface $R/\mathbb{P}^1$. 
%The involution $\psi$ lifts to $R$ which acts on the base trivially. 
%Hence, by choosing a zero-section, it corresponds to a translation by a $2$-torsion section $\sigma$. 
%We can then easily check that the resulting Enriques surface $Y$ and the involution $\iota$ is nothing but 
%the quadratic twist discussed in Subsection \ref{fibrations}. The case corresponds to two smooth fibers (namely
%$F_i=E_i\ (i=1,2)$), hence by the discussion there we see that this belongs to No. [11].
Then $E_i$ are smooth elliptic curves and preserved by $\psi$ (Figure~\ref{11}). 

Blow up $\mathbb{P}^1\times\mathbb{P}^1$ at $8$ intersection points of $E_1$ and $E_2$. 
Let $R$ be the blown up surface. 
We denote by $E_1'$ and $E_2'$ the strict transforms of $E_1$ and $E_2$ respectively. 
Let $B'=E_1'+E_2'$. 
The $K3$ surface $X$ is the double cover of $R$ whose branch locus is $B'$. 
Since $X^{\theta}=B'$ consists of two elliptic curves, we see $(r, l, \delta)=(10, 8, 0)$, by Theorem~\ref{NikThm422}. 
To see to which No. this example belongs, we argue as follows.

The involution $\psi$ of $\mathbb{P}^1\times\mathbb{P}^1$ lifts to the rational elliptic surface $R/\mathbb{P}^1$, which acts on the base trivially.
Hence, by choosing a zero-section, it corresponds to a translation by a $2$-torsion section $\sigma$. 
In this case, the Horikawa construction corresponds exactly to the quadratic twist construction discussed in \cite{Kon,HS}:
the free involution $\varepsilon$ is given by a lift of the translation automorphism by $\sigma$. 
We remark that generically the elliptic surface $R$ has eight singular fibers $4I_2+4I_1$ (Kodaira's notation). 

Here we consider a deformation of the $K3$ surface $X$: 
we move the branch locus $B'$ to $B'_1$, the union of one $I_2$ fiber plus one smooth fiber. We denote by $X_1$ the 
smooth $K3$ surface obtained by the double cover branched along $B'_1$ and the minimal desingularization.
Since only rational double points appear in construction, $X$ and $X_1$ are connected by a smooth deformation.
Now $X_1$ has also an Enriques quotient $Y_1$ via the quadratic twist construction. By definition of $B'_1$, 
the main invariant of $\theta_1$ on $X_1$ is $(12,8,1)$ and the 
associated involution on $Y_1$ has type No.~[8]. 
We recall that a specialization of $K3$ surfaces $X \rightsquigarrow X_1$ exists if and only if $T_{X_1}\subset T_X$. 
Hence we see that our example belongs to No.~[11].

\begin{figure}[htbp]
 \begin{minipage}{0.45\hsize}
  \begin{center}
   \includegraphics[width=40mm]{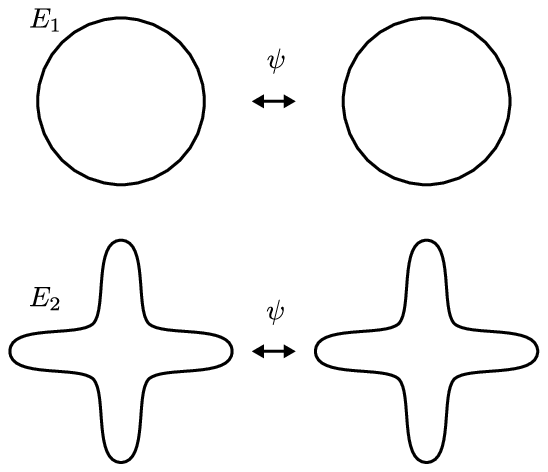}
  \end{center}
  \caption{}
  \label{11}
 \end{minipage}
 \begin{minipage}{0.45\hsize}
  \begin{center}
   \includegraphics[width=40mm]{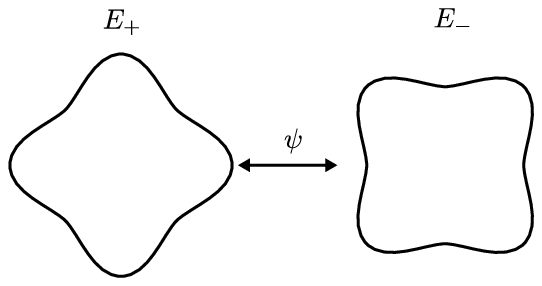}
  \end{center}
  \caption{}
  \label{12}
 \end{minipage}
\end{figure}

\medskip
\noindent
\textbf{Example No.~[12].}\quad 
%Consider the following curves on $\mathbb{P}^1\times\mathbb{P}^1$ (Figure~); 
Consider the following curves on $\mathbb{P}^1\times\mathbb{P}^1$; 
\[E_{\pm}\colon v^2(u^2\pm a_1 u+a_2)\pm v(u^2\pm a_3 u+a_4)+(u^2\pm a_5 u+a_6)=0 \pod{a_i\in\mathbb{C}}. \]
%Then $E_i$ are elliptic curves which are exchanged by $\psi$. 
%We can see that $X^{\theta}$ consists of strict transforms of $E_i$ and 
%therefore $(r,l,\delta )=(10,8,0)$. To check that they correspond to No.~[12] in this case, we discuss as follows.
Then $E_{\pm}$ are elliptic curves which are exchanged by $\psi$ (Figure~\ref{12}). 

Blow up $\mathbb{P}^1\times\mathbb{P}^1$ at $8$ intersection points of $E_{\pm}$. 
Let $R$ be the blown up surface. 
We denote by $E_{\pm}'$ the strict transforms of $E_{\pm}$ respectively. 
Let $B'=E_+'+E_-'$. 
The $K3$ surface $X$ is the double cover of $R$ whose branch locus is $B'$. 
Since $X^{\theta}=B'$ consists of two elliptic curves, we see $(r, l, \delta)=(10, 8, 0)$, by Theorem~\ref{NikThm422}. 
To check that they correspond to No.~[12] in this case, we discuss as follows.

We remark that the case No.~[9] is a specialization of our family: it is exactly the case where $E_{\pm}$ acquire nodes.
By simultaneous resolution, we can regard the $K3$ surface $X_0$ in No.~[9] as a special member of a 
smooth deformation with general fiber $X_1$ from our family. Here,
the two elliptic curves $E_{\pm}$ deform into the sums of two rational curves $F_{\pm}+F_{\pm}'$, where 
$(F_{\pm}^2)=((F_{\pm}')^2)=-2$ and $(F_{\pm},F_{\pm}')=2$ (double sign corresponds). 

Moreover, since the formation of $\varepsilon$ does not change under this specialization, our family is in fact 
a family of $K3$ surfaces with free involutions $(X_1,\varepsilon_1)$ and $(X_0,\varepsilon_0)$. (In other words,
the free involutions are preserved under the specialization.)
By the theory of period maps, we have an inclusion $NS(X_1)\subset NS(X_0)$. The orthogonal complement 
is generated by the $(-4)$-vector $F_+-F_-$, and the overlattice structure is given by
\[F_+=\frac{F_++F_-}{2}+\frac{F_+-F_-}{2}\in NS(X_0).\]
Hence, we can compute $\det NS(X_0)=\det NS(X_1)\cdot 4/2^2=\det NS(X_1)$. Recalling that 
$\det NS$ is the same as $\det K_-$ in each case, we can see that our example belongs to No.~[12].

\medskip
\noindent
\textbf{Example No.~[14].}\quad 
We need an irreducible curve on $\mathbb{P}^1\times\mathbb{P}^1$ which has $8$ nodes and stable under $\psi$, but
it seems not easy to construct them in a direct way.
The following construction is due to H.~Tokunaga. 

Let $B_0$ be a smooth irreducible divisor of bidegree $(2,2)$ to which the four lines $u=0,\infty$; $v=0,\infty$ 
are tangent. We remark that in general, if a divisor is tangent to the branch curve (with local intersection number $2$), then 
by pulling back to the double cover, the divisor acquires a node at the point of tangency. Thus in our case the following construction works:
We consider the two self-morphisms $\psi_1\colon (u,v)\mapsto (u^2,v)$ and $\psi_2\colon (u,v)\mapsto (u,v^2)$ 
of $\mathbb{P}^1\times\mathbb{P}^1$. Then, the pullback $C_8:=(\psi_1\circ \psi_2)^* (B_0)$ has bidegree $(4,4)$ with 
eight nodes and is stable under $\psi$ (Figure~\ref{14}).

\begin{figure}[htbp]
  \begin{center}
   \includegraphics[height=40mm]{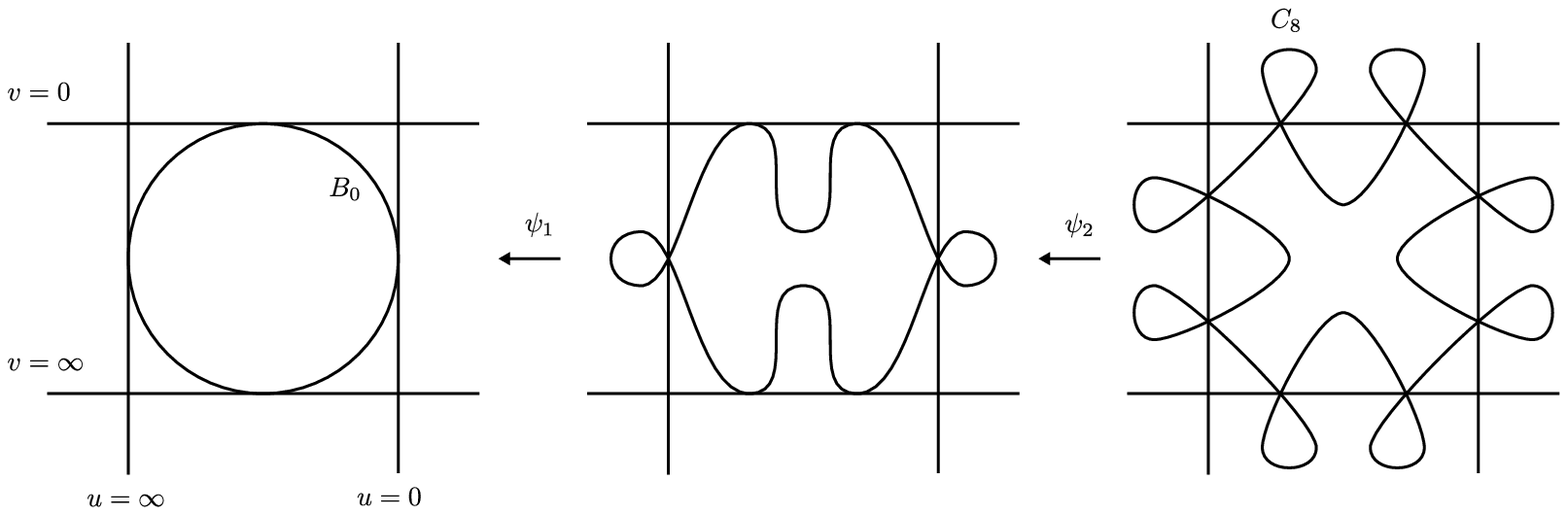}
  \end{center}
  \caption{}
  \label{14}
\end{figure}

We can exhibit the equation for $C_8$ as follows, for example.
\[(c^2u^4+2cbu^2+b^2)v^4+(2cau^4+du^2+2b)v^2+(a^2u^4+2au^2+1)=0. \]

Blow up $\mathbb{P}^1\times\mathbb{P}^1$ at $8$ nodes of $C_8$. 
Let $R$ be the blown up surface. 
We denote by $C_8'$ the strict transforms of $C_8$. 
The $K3$ surface $X$ is the double cover of $R$ whose branch locus is $C_8'$. 
Since $X^{\theta}=C_8'$ is a elliptic curve, we see $(r, l, \delta)=(10, 10, 1)$, by Theorem~\ref{NikThm422}. 
%By $r=10$, we see $\rank S_-=4$. 
%Hence we see $S_{\pm}=A_1(2)^4$ by Lemma~\ref{ParityOfT}. 
Therefore this is the example of No.~[14]. 

\medskip
\noindent
\textbf{Example No.s~[15]--[18].}\quad 
Let $C_{2i} \pod{i=0, 1, 2, 3}$ be irreducible curves on $\mathbb{P}^1\times\mathbb{P}^1$ whose bidegree is $(4, 4)$ with $2i$ nodes respectively. 

Blow up $\mathbb{P}^1\times\mathbb{P}^1$ at $2i$ nodes of $C_{2i}$. 
Let $R_{2i}$ be the blown up surface. 
We denote by $C_{2i}'$ the strict transforms of $C_{2i}$. 
The $K3$ surface $X_{2i}$ is the double cover of $R_{2i}$ whose branch locus is $C_{2i}'$. 
Since $X_{2i}^{\theta}=C_{2i}'$ is a curve of genus $9-2i$, we see $(r, l)=(2i+2, 2i+2)$, by Theorem~\ref{NikThm422}. 
%By $r=2i+2$, we see $\rank S_-=i$. 
%Therefore $X_6$, $X_4$, $X_2$ and $X_0$ are the examples of No.~[15], [16], [17] and [18] respectively, by Theorem~\ref{Our lists}. 
Therefore the cases $i=3$, $2$, $1$ and $0$ are the examples of No.~[15], [16], [17] and [18] respectively. 

\subsection{Enriques' sextics}\label{sextic}

The non-normal sextic surface in $\mathbb{P}^3$ which is singular along the six edges of a tetrahedron is 
a model of Enriques surface, the one first considered by Enriques himself. 
In fact its normalization gives a smooth Enriques surface, see \cite{GH}. 
Setting the tetrahedron as $xyzt=0$, the general equation of such surfaces is given by 
\[q(x,y,z,t)xyzt+(x^2y^2z^2+x^2y^2t^2+x^2z^2t^2+y^2z^2t^2)=0,\]
where $q$ is a quadratic equation. By considering various linear actions on $\mathbb{P}^3$, 
we can get many examples of involutions on Enriques surfaces. The most important for us among them
is the following example exhibiting No.~[13]. 

\medskip
\noindent
\textbf{Example No.~[13].}\quad 
Let us consider the involution $\iota \colon (x:y:z:t)\mapsto (y:x:t:z)$ on $\mathbb{P}^3$. 
The general equation of invariant Enriques' sextic $\overline{Y}$ looks as 
\begin{equation*}
\begin{split}
\left( a_1(x^2+y^2)+  a_2(z^2+t^2)+  a_3 xy+  a_4 zt+   a_5(xz+yt)+   a_6(xt+yz)\right) xyzt\\
+ (x^2y^2z^2+x^2y^2t^2+x^2z^2t^2+y^2z^2t^2)=0,
\end{split}
\end{equation*}
where $a_i\in \mathbb{C}$ are general. Then the normalization $Y$ is a smooth Enriques surface 
with the induced action by $\iota$. 

Let us show that they belong to No.~[13]. 
Since in this case $\theta$ is also fixed-point-free, 
this is equivalent to saying that the fixed locus $Y^{\iota}$ is a finite set. Moreover
since the normalization $Y\rightarrow \overline{Y}$ is a finite morphism, it suffices to show that $\overline{Y}^{\iota}$ is 
a finite set. But this set is the intersection of $\overline{Y}$ with the fixed locus in $\mathbb{P}^3$, 
$\{x=y,z=t\}\cup \{x+y=0,z+t=0\}$. Since the general element does not contain these lines, 
the intersection is a finite set as desired.

\end{document}